\documentclass[oneside,english]{amsart}
\usepackage[T1]{fontenc}
\usepackage[latin9]{inputenc}
\usepackage{float}
\usepackage{mathrsfs}
\usepackage{amsbsy}
\usepackage{amstext}
\usepackage{amsthm}
\usepackage{amssymb}
\usepackage{graphicx}

\makeatletter
\numberwithin{equation}{section}
\numberwithin{figure}{section}
  
  \theoremstyle{plain}
  \newtheorem*{thm*}{\protect\theoremname}
  \newtheorem{thm}{\protect\theoremname}[section]
  \newtheorem{lem}[thm]{\protect\lemmaname}
  \newtheorem{cor}[thm]{\protect\corollaryname}

  \theoremstyle{definition}
  \newtheorem{defn}[thm]{\protect\definitionname}
  \newtheorem{example}[thm]{\protect\examplename}
  
  \theoremstyle{remark}
  \newtheorem{rem}[thm]{\protect\remarkname}

  \newcommand{\R}{{\mathbb {R}}}
    \newcommand{\N}{{\mathbb {N}}}
  
\date{}
\usepackage{ifpdf}
\ifpdf

\IfFileExists{lmodern.sty}
 {\usepackage{lmodern}}{}

\fi 
\usepackage[a4paper, hmargin=2.5cm, vmargin={3.2cm, 3.2cm}]{geometry}

\author{Yotam Smilansky} 

\address{Yotam Smilansky, Raymond and Beverly Sackler School of Mathematical Sciences, Tel Aviv University, Tel Aviv 69978, Israel$^§$. \medskip \newline \textit{Email address}: {\tt yotam.smilansky@rutgers.edu} \bigskip \newline \textit{$^§$Current address}: Department of Mathematics, Rutgers University, 110 Frelinghuysen Road, Piscataway, NJ 08854, USA.}

\makeatother

\usepackage{babel}
  \providecommand{\corollaryname}{Corollary}
  \providecommand{\definitionname}{Definition}
  \providecommand{\examplename}{Example}
  \providecommand{\lemmaname}{Lemma}
  \providecommand{\remarkname}{Remark}
  \providecommand{\theoremname}{Theorem}
\providecommand{\theoremname}{Theorem}

\begin{document}
\global\long\def\adj{{\rm adj}}

\global\long\def\tr{{\rm tr}}
\global\long\def\rank{{\rm rank}}

\global\long\def\span{{\rm span}}

\global\long\def\diag{{\rm diag}}

\global\long\def\vol{{\rm vol}}
\global\long\def\supp{{\rm supp}}
\global\long\def\mod{\text{mod}}

\title{Uniform Distribution of Kakutani Partitions Generated By Substitution
Schemes}
\begin{abstract}
Substitution schemes provide a classical method for constructing tilings
of Euclidean space. Allowing multiple scales in the scheme, we introduce
a rich family of sequences of tile partitions generated by the substitution
rule, which include the sequence of partitions of the unit interval
considered by Kakutani as a special case. Using our recent path counting
results for directed weighted graphs, we show that such sequences
of partitions are uniformly distributed, thus extending Kakutani's
original result. Furthermore, we describe certain limiting frequencies
associated with sequences of partitions, which relate to the distribution
of tiles of a given type and the volume they occupy.
\end{abstract}

\maketitle

\section{Introduction and main results}

\subsection{The Kakutani splitting procedure}

The splitting procedure introduced by Kakutani in \cite{Kakutani}
and the associated sequences of partitions $\left\{ \pi_{m}\right\} $
of the unit interval $\mathcal{I}=\left[0,1\right]$ are defined as
follows: Fix a constant $\alpha\in\left(0,1\right)$ and define the
first element of the sequence to be the trivial partition $\pi_{0}=\mathcal{I}$.
Next, define $\pi_{1}$ to be the partition of $\mathcal{I}$ into
two subintervals $\left[0,\alpha\right]$ and $\left[\alpha,1\right]$,
that is, into two intervals of disjoint interior, one of length $\alpha$
and one of length $1-\alpha$. Assuming $\pi_{m-1}$ is defined, define
the partition $\pi_{m}$ by splitting every interval of maximal length
in $\pi_{m-1}$ into two subintervals with disjoint interiors, proportional
to the two subintervals that constitute $\pi_{1}$. This sequence
is known as the $\alpha$-Kakutani sequence of partitions, and the
procedure that generates it is called the Kakutani splitting procedure. 

For example, the trivial partition $\pi_{0}$ and the first four non-trivial
partitions $\pi_{1},\pi_{2},\pi_{3}$ and $\pi_{4}$ of the $\frac{1}{3}$-Kakutani
sequence of partitions of the unit interval $\mathcal{I}$ are illustrated
in Figure \ref{fig:a few elements in Kakutani third partition}. 

\begin{figure}[H]
	\includegraphics[scale=1.56]{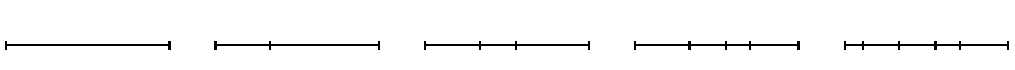}\caption{\label{fig:a few elements in Kakutani third partition}From left to
		right, the first few partitions in the $\frac{1}{3}$-Kakutani sequence
		of partitions of the unit interval.}
\end{figure}

A sequence $\left\{ \pi_{m}\right\} $ of partitions of $\mathcal{I}$
is said to be uniformly distributed if for any continuous function
$f$ on $\mathcal{I}$ 
\[
\lim_{m\rightarrow\infty}\frac{1}{k(m)}\sum_{i=1}^{k(m)}f\left(t_{i}^{m}\right)=\int\limits _{\mathcal{I}}f(t)dt,
\]
where $k(m)$ is the number of intervals in the partition
$\pi_{m}$ and $t_{i}^{m}$ is the right endpoint of the interval $i$
in the partition $\pi_{m}$. We remark that the choice of sampling
points as the right endpoints of the intervals is arbitrary. 
\begin{thm*}[Kakutani]
	For any $\alpha\in\left(0,1\right)$, the $\alpha$-Kakutani sequence
	of partitions is uniformly distributed.
\end{thm*}
Our main result is a proof of uniform distribution for sequences of
partitions generated by multiscale substitution schemes in $\mathbb{R}^{d}$,
constituting a generalization of the $\alpha$-Kakutani sequences
considered above.

\subsection{\label{sec:Multiscale-Substitution-Schemes}Tiles, partitions and multiscale substitution schemes}

\begin{defn}
	A \textit{tile} $T\subset\mathbb{R}^{d}$ is a Lebesgue measurable
	bounded set with positive Lebesgue measure and boundary of measure
	zero. The measure of $T$ is denoted by $\vol T$ and referred to
	as the volume of $T$.
\end{defn}

\begin{defn}
	Let $U\subset\mathbb{R}^{d}$ be a bounded set. A \textit{partition} of $U$ is a set of subsets of $U$ with pairwise disjoint interiors, the union
	of which is equal to $U$. Given a list
	of tiles $L$, we say that $L$ \textit{tiles} $U$, and that $U$ is \textit{partitioned} into elements of $L$, if there exists a partition of $U$ such that each element of the partition is isometric to exactly one tile in $L$. 
\end{defn}

\begin{defn}\label{def: multicale substitution scheme}
	A \textit{multiscale substitution scheme} $\sigma=(\tau_\sigma,\omega_\sigma,\Sigma_\sigma)$ in $\R^d$ consists of a finite list $\tau_\sigma=(T_1,\ldots,T_n)$ of labeled tiles in $\R^d$  called \textit{prototiles}, each equipped with a finite list of \textit{substitution tiles} 
	\begin{equation}
	\omega_\sigma(T_i)=\left(\alpha_{ij}^{\left(k\right)} T_j:\,j=1,\ldots,n,\,\,k=1,\ldots,k_{ij}\right)\notag
	\end{equation} 
	so that $\omega_\sigma(T_i)$ tiles $T_i$, and a family $\Sigma_\sigma$ of  \textit{substitution rules}, where each $\varrho_\sigma\in\Sigma_\sigma$ assigns a partition $\varrho_\sigma(T_i)$ of $T_i$ into elements of $\omega_\sigma(T_i)$. By assumption $\Sigma_\sigma$ is non-empty.
	
	\noindent For all $i,j$ and $k$ the scaling constants $\alpha_{ij}^{(k)}$ are positive, and the \textit{constants of substitution}
	\begin{equation}
	\beta_{ij}^{(k)}=\left(\frac{\vol T_{j}}{\vol T_{i}}\right)^{1/d}\alpha_{ij}^{(k)}\label{eq: constants of substitution}
	\end{equation}
	satisfy $0<\beta_{ij}^{(k)}<1$.  
\end{defn}

It is convenient to think of the analogy to jigsaw puzzles, where the prototiles $\tau_\sigma$ are puzzles to be solved using the pieces in $\omega_\sigma$, and every $\varrho_\sigma\in\Sigma_\sigma$ gives a solution $\varrho_\sigma(T_i)$ to each of the puzzles. 
\begin{figure}[H]
	\includegraphics[scale=1.1]{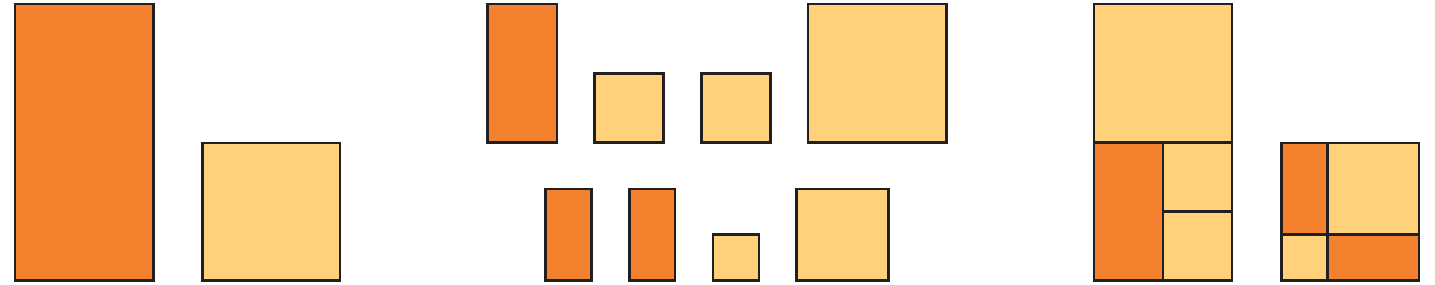}\caption{\label{fig: RS scheme}A substitution scheme $\sigma$ on a rectangle $\mathcal{R}$ and a square $\mathcal{S}$ in $\R^2$. From left to right, the prototiles $\tau_\sigma=(\mathcal{R},\mathcal{S})$, the substitution tiles $\omega_\sigma(\mathcal{R})$ (top) and  $\omega_\sigma(\mathcal{S})$, and a substitution rule $\varrho_\sigma$.}
\end{figure}
A set of prototiles $\tau_\sigma$ may contain two or more labeled tiles that are identical as sets in $\mathbb{R}^{d}$, or more
generally are the images of each other under conformal maps of $\mathbb{R}^{d}$, but carry distinct labels. 
When applying a rescaling or an isometry to a labeled tile, we assume that the label is preserved. This allows us to naturally extend substitution rules to the image of a tile under such mappings. More precisely, given a tile $T=\alpha\cdot\varphi (T_i)$, where $\alpha>0$ and $\varphi$ is an isometry of $\R^d$, for any $\varrho_\sigma\in\Sigma_\sigma$ we define
\[
\varrho_\sigma(T)=\varrho_\sigma(\alpha\cdot\varphi\left( T_{i}\right))=\alpha\cdot\varphi\left(\varrho_\sigma(T_{i})\right)
\]
and
\[
\omega_\sigma(T)=\omega_\sigma(\alpha\cdot\varphi\left(T_{i}\right))=\left(\alpha T'\,:\, T'\in \omega_\sigma(T_i)\right).
\] 
For every $k\in\N$ we can thus define inductively 
\[
\varrho_\sigma^{k+1}(T)=\bigsqcup_{T'\in\varrho_\sigma^{k}(T)}\varrho_\sigma(T')\,\,\,\,\,\,\,\text{and}\,\,\,\,\,\,\,\omega_\sigma^{k+1}(T)=\left(T''\in\omega_\sigma(T')\,:\,T'\in\omega_\sigma^k(T)\right).
\]

\begin{defn}\label{def: all substitution tiles}
	\label{def:Irresucible multiscale substitution}Let $\sigma$
	be a multiscale substitution scheme. Denote by
	\[
	\mathscr{O}^\sigma_{i}=\left(T\in \omega_\sigma^k(T_i)\,:\,k\in\N\right)
	\] 
	the list of substitution tiles defined by finitely many applications of elements of $\Sigma_\sigma$ to $ T_{i}$. A tile $T\in\mathscr{O}^\sigma_{i}$ is \textit{of type} $j$ if it is labeled $j$, and $\sigma$
	is \textit{irreducible} if $\mathscr{O}^\sigma_{i}$ contains a tile of type
	$j$ for any $1\leq i,j\leq n$.
\end{defn}

\subsection{Sequences of partitions generated by multiscale substitutions schemes}

\begin{defn}
	Let $\sigma$ be a multiscale substitution scheme.
\begin{enumerate}
	\item A \textit{Kakutani sequence of partitions} $\left\{ \pi_{m}\right\} $
	of $ T_{i}\in\tau_\sigma$ is defined as follows: \\
	The trivial partition is $\pi_{0}= T_{i}$. For any $m\in\mathbb{N}$,
	if $\pi_{m-1}$ is a partition of $ T_{i}$ consisting of
	labeled tiles, then the partition $\pi_{m}$ is defined by substituting tiles of maximal volume in $\pi_{m-1}$, each tile substituted according to an arbitrarily and independently chosen substitution rule $\varrho_\sigma\in\Sigma_\sigma$.
	\item A \textit{generation sequence of partitions} $\left\{ \delta_{k}\right\} $
	of $ T_{i}\in\tau_\sigma$ is defined as follows: \\
	The trivial partition is $\delta_{0}= T_{i}$. For any $k\in\mathbb{N}$,
	if $\delta_{k-1}$ is a partition of $ T_{i}$ consisting of
	labeled tiles, then the partition $\delta_{k}$ is defined by substituting
	all tiles in $\delta_{k-1}$, each tile substituted according to an arbitrarily and independently chosen substitution rule $\varrho_\sigma\in\Sigma_\sigma$. 
\end{enumerate}
\end{defn}

For illustrated examples of Kakutani and generation sequences of partitions, see Section \ref{sec: examples}. Note that for any Kakutani sequence $\{\pi_m\}$ and generation sequence $\{\delta_k\}$ of $T_i$, every tile $T\in\mathscr{O}^\sigma_{i}$ corresponds to at least one partition $\pi_{m}$ and to exactly one partition $\delta_{k}$, and that $\mathscr{O}^\sigma_{i}$ is the union of all substitution tiles used to tile the elements of $\{\pi_m\}$  and $\{\delta_k\}$. 

\begin{rem}
	In the terminology of \cite{Kakutani}, the splitting procedures that define $\left\{ \pi_{m}\right\} $
	and $\left\{ \delta_{k}\right\} $ correspond to ``maximal refinement''
	and ``refinement'', respectively, and the generation of a
	tile is its ``rank''.
\end{rem}

\subsection{Uniform distribution}
\begin{defn}
	\label{def: ud sequences of sets of points}Let $U\subset\mathbb{R}^{d}$
	be a bounded measurable set of positive measure, and for every $m\in\mathbb{N}$
	let $x_{m}$ be a finite set of points in $U$ of cardinality $\left|x_{m}\right|$
	tending to infinity with $m$. The sequence $\left\{ x_{m}\right\} $
	is \textit{uniformly distributed }in $U$ if for any continuous function
	$f$ on $U$ 
	\[
	\lim_{n\rightarrow\infty}\frac{1}{\left|x_{m}\right|}\sum_{x\in x_{m}}f(x)=\frac{1}{\vol U}\int\limits _{U}f(t)dt,
	\]
	where the integration is with respect to Lebesgue measure.
\end{defn}

\begin{defn}
	Let $\left\{ \gamma_{m}\right\} $ be a sequence of partitions of
	$U$. A\textit{ marking sequence} $\left\{ x_{m}\right\} $ of $\left\{ \gamma_{m}\right\} $
	is a sequence of sets of points in $U$, such that every set in the
	partition $\gamma_{m}$ contains a single point of $x_{m}$, and all
	points in $x_{m}$ are distinct. The sequence of partitions $\left\{ \gamma_{m}\right\} $
	is \textit{uniformly distributed} if there exists a marking sequence
	$\left\{ x_{m}\right\} $ of $\left\{ \gamma_{m}\right\} $ that
	is uniformly distributed in $U$.
\end{defn}
Definition \ref{def: ud sequences of sets of points} is equivalent
to the weak-{*} convergence of the normalized sampling measures
\[
\frac{1}{\left|x_{m}\right|}\sum_{x\in x_{m}}\delta_{x}
\]
to the normalized Lebesgue measure on $U$, where $\delta_{x}$ is
the Dirac measure concentrated at $x$. The marking introduced in
\cite{Kakutani} consists of the boundary points of the intervals
constituting each partition of $\mathcal{I}$. \\

The following is our main result.
\begin{thm}
	\label{Thm: Main result}Let $\sigma$ be an irreducible multiscale substitution scheme, and let $\left\{ \pi_{m}\right\} $ be a Kakutani
	sequence of partitions of $T_{i}\in\tau_\sigma$. Then $\left\{ \pi_{m}\right\} $ is uniformly distributed in $ T_{i}$.
\end{thm}

The irreducibility condition is crucial. As shown in Section 2 of
\cite{Volcic}, there are simple examples of sequences of partitions
generated by non-irreducible schemes which are not uniformly distributed. 
\begin{rem}
	In our proof of Theorem \ref{Thm: Main result} as well as in the
	proofs of the results stated below, the choice of marking sequences
	is arbitrary. In fact, we show that every marking sequence of the
	sequences of partitions is uniformly distributed.
\end{rem}

\subsection{Incommensurable schemes and frequencies of types }
\begin{defn}\label{def: incommensurable schemes}
	An irreducible multiscale substitution scheme
	$\sigma$ is \textit{incommensurable} if there exist $1\leq i,j\leq n$
	and two tiles $T_{1}\in\mathscr{O}^\sigma_{i}$ of type $i$ and $T_{2}\in\mathscr{O}^\sigma_{j}$ of type $j$ so that
	\[
	\log\frac{\vol T_{1}}{\vol T_{i}}\notin\mathbb{Q}\log\frac{\vol T_{2}}{\vol T_{j}}.
	\]
	Otherwise the scheme is \textit{commensurable}.
\end{defn}

Easy non-trivial examples arise from $\alpha$-Kakutani sequences, which can be formulated in the language of multiscale substitution schemes with a prototile set consisting only of the unit interval, and a substitution rule partitioning the unit interval into the union of the intervals $[0,\alpha]$ and $[\alpha,1]$, see also Example \ref{ex:the alpha-Kakutani construction}. If $\alpha=\frac{1}{3}$, we can simply pick $T_1=[0,\frac{1}{3}]$ and $T_2=[\frac{1}{3},1]$, and since $\frac{\log 3}{\log 2}$ is irrational, incommensurability follows. On the other hand, since all intervals that appear in an $\alpha$-Kakutani sequence are of length $\alpha^k(1-\alpha)^\ell$, if $\varphi$ is the golden ratio and $\alpha=\frac{1}{\varphi}$, then since $1-\frac{1}{\varphi}=\frac{1}{\varphi^2}$ all interval lengths are an integer power of $\varphi$, and commensurability follows.

Admittedly, the definition of incommensurability strictly in terms of the substitution scheme $\sigma$ seems rather mysterious. As we will see, it is more natural and easier to verify when considered in the context of the graph $G_\sigma$ associated with the substitution scheme $\sigma$, introduced in Section \ref{sec:Graphs}. For now we only note that incommensurable schemes are generic in the sense that for almost any choice of constants of substitution, the resulting scheme is incommensurable. We also refer to Section \ref{sec:Incommensurable and commensurable and examples} for equivalent definitions and additional examples of commensurable and incommensurable substitution schemes.
\begin{defn}
	Let $\sigma$
	be a multiscale substitution scheme and let $\left\{ \gamma_{m}\right\} $ be either a Kakutani or a generation sequence of partitions of $ T_{i}\in\tau_\sigma$. For
	$1\leq r\leq n$, an \textit{$r$-marking sequence} $\{ x_{m}^{\left(r\right)}\} $
	of $\left\{ \gamma_{m}\right\} $ is a sequence of sets of points
	in $ T_{i}$, such that every tile of type $r$ in the partition
	$\gamma_{m}$ contains a single point of $x_{m}^{\left(r\right)}$,
	and all points in $x_{m}^{\left(r\right)}$ are distinct.
\end{defn}

\begin{thm}
	\label{thm:Types, u.d and frequencies}Let $\sigma$
	be an irreducible incommensurable multiscale substitution scheme and let $\left\{ \pi_{m}\right\} $ be a Kakutani
	sequence of partitions of $ T_{i}\in\tau_\sigma$. Let $1\leq r\leq n$. 
	\begin{enumerate}
		\item Any $r$-marking sequence $\{ x_{m}^{\left(r\right)}\} $
		of $\left\{ \pi_{m}\right\} $ is uniformly distributed in $ T_{i}$. 
		\item The ratio between the number of tiles of type $r$ in $\pi_{m}$ and
		the total number of tiles is
		\[
		\frac{\sum\limits _{h=1}^{n}q_{h}\sum\limits _{k=1}^{k_{hr}}\left(1-\left(\beta_{hr}^{(k)}\right)^{d}\right)}{\sum\limits _{h,j=1}^{n}q_{h}\sum\limits _{k=1}^{k_{hj}}\left(1-\left(\beta_{hj}^{(k)}\right)^{d}\right)}+o\left(1\right),\quad m\rightarrow\infty.
		\]
		\item The volume of the region covered by tiles of type $r$ in $\pi_{m}$
		is
		\[
		\vol T_{i}\cdot\sum_{h=1}^{n}q_{h}\sum_{k=1}^{k_{hr}}\left(\left(\beta_{hr}^{(k)}\right)^{d}\log\frac{1}{\beta_{hr}^{(k)}}\right)+o\left(1\right),\quad m\rightarrow\infty.
		\]
	\end{enumerate}
	In the equations above $q_{h}$ is any entry of column $h$ of a matrix $Q_\sigma\in M_{n}(\mathbb{R})$
	of equal rows, given by 
	\[
	Q_\sigma=\frac{\adj\left(I-M_\sigma\right)}{-\tr\left(\adj\left(I-M_\sigma\right)\cdot M_\sigma^{\prime}\right)}
	\]
	with $\left(M_\sigma\right)_{ij}=\sum_{k=1}^{k_{ij}}\left(\beta_{ij}^{(k)}\right)^{d}$
	and $\left(M_\sigma^{\prime}\right)_{ij}=\sum_{k=1}^{k_{ij}}\left(\beta_{ij}^{(k)}\right)^{d}\log\beta_{ij}^{(k)}$.
\end{thm}

For Kakutani sequences generated by commensurable schemes, which are
described in Section \ref{sec:Incommensurable and commensurable and examples}
and include schemes of fixed scale, Theorem \ref{thm:Types, u.d and frequencies}
does not necessarily hold, and the limits appearing in it may not
even exist, see Example \ref{example:penrose statistics}.

\subsection{Generation sequences of partitions}

Generation sequences of partitions are generally not uniformly distributed. For example the generation
sequence generated by the $\alpha$-Kakutani multiscale substitution
scheme is not uniformly distributed for any $\alpha\not=\frac{1}{2}$, see also Example \ref{ex:the alpha-Kakutani construction}.

\begin{defn}
	A multiscale substitution scheme is \textit{fixed scale}
	if there exists $\alpha\in\left(0,1\right)$ so that
	\[
	\alpha_{ij}^{(k)}=\alpha
	\]
	for all $i,j$ and $k$. In this case $\alpha$ is called the \textit{contraction constant}.
\end{defn}

Clearly, fixed scale substitution schemes are commensurable, since for every $1\le i \le n$ and every tile $T$ of type $i$ in $\mathscr{O}^\sigma_i$, the ratio $\frac{\vol T}{\vol T_i}$ is an integer power of $\alpha$.
\begin{rem}
	Fixed scale substitution schemes are the classical setup of substitution
	tilings.
\end{rem}

\begin{thm}
	\label{Thm: main result for generation sequences}Let $\sigma$
	be an irreducible fixed scale substitution scheme and let $\left\{ \delta_{k}\right\} $ be a
	generation sequence of partitions of $ T_{i}\in\tau_\sigma$.
	Then $\left\{ \delta_{k}\right\} $ is uniformly distributed in $ T_{i}$.
\end{thm}

The counterpart of Theorem \ref{thm:Types, u.d and frequencies} for
generation sequences generated by fixed scale schemes do not necessarily
hold without an additional assumption of primitivity, as demonstrated by
Corollary \ref{cor: generation of non-primitive is a union of generation}.
In Section \ref{sec: sequences generated by fixed scaled} we prove
the special case concerning primitive schemes, see Theorem \ref{Thm: frequencies for generation fixed scale}.

\subsection{Additional background and related topics}

Kakutani's original proof of uniform distribution of the $\alpha$-Kakutani
sequences given in \cite{Kakutani} as well as Adler and Flatto's
proof given in \cite{Adler Flatto}, both use ergodic and measure theoretical
tools and are of a different nature than the proof presented here.
Various generalizations of the Kakutani splitting procedure have been
studied, and in recent years the procedure has been generalized in
two directions. In \cite{CarboneVolcic}, Carbone and Vol{\v c}i{\v c}
define a generalization of the splitting procedure which generates
sequences of partitions of higher dimensional sets. Vol{\v c}i{\v c}
also studies in \cite{Volcic} an extended Kakutani one dimensional
splitting procedure in which $\pi_{1}$ consists of any finite number
of subintervals of $\mathcal{I}$, with subsequent partitions defined
by composing the same splitting rule with a dilation. In both cases
the resulting sequences of partitions are shown to be uniformly distributes,
and while the higher dimensional construction introduced in \cite{CarboneVolcic}
is different from the one presented here, the extended one dimensional
construction can be interpreted as a multiscale substitution scheme
on a single prototile in $\mathbb{R}^{1}$. This construction is further
studied in \cite{Aistleitner Hofer}, \cite{Drmota Infusino} and
in \cite{Infusino}, which also contains a survey of results on
uniform distribution of points and partitions.

Multiscale substitution schemes and their associated graphs were originally
considered by the author in a joint work in progress with Yaar Solomon
\cite{SmiSol} concerning the study of a new family of tilings
of the Euclidean space. Denote by $\mathcal{X}$ the collection of
closed sets in $\mathbb{R}^{d}$, then equipped with an appropriate
topology $\mathcal{X}$ is compact. The multiscale scheme is used
to define tilings of bounded sets of growing diameter in $\mathbb{R}^{d}$,
which constitute a sequence of closed sets in $\mathcal{X}$. Tilings
of the entire space are defined as partial limits of such sequences,
and the space of all tilings of $\mathbb{R}^{d}$ constructed this
way is denoted by $\mathcal{X}_{H}$, which is a compact subspace
of $\mathcal{X}$. Our forthcoming paper contains more information
about the tilings $\tau\in\mathcal{X}_{H}$ themselves as well as
the tiling space $\mathcal{X}_{H}$. Examples include Sadun's generalizations of the pinwheel tiling presented in \cite{Sadun}, see also Example 5.8 and Remark 5.9 below. In addition, the construction described in the Section A.5 in Appendix A of \cite{Fusion} is closely related to the $\frac{1}{3}$-Kakutani scheme given in Example 2.1 below, and is studied there in the context of fusion tilings of infinite local complexity. 

This work is inspired by the theory of aperiodic tilings of Euclidean
space, and more specifically by the hierarchical construction known
as substitution, which is used to generate some well known aperiodic
tilings, most famously the Penrose tiling. We visit in examples some
well known substitutions, including the generalized pinwheel \cite{Sadun},
the Rauzy fractal \cite{Rauzy} and the Penrose-Robinson
substitution, see \cite{BaakeGrimm} for more about Penrose-type
constructions, substitutions and other aperiodic tilings. This work is also related to the construction and study of self-similar fractal strings \cite{Lapidus fractals} and of graph-directed fractal sprays \cite{Turkish fractal guys}.

\subsection{Acknowledgments }

I would like to thank Zemer Kosloff, Yaar Solomon, Anna Vaskevich
and Barak Weiss for their valuable contribution to this work, and
Avner Kiro, Uzy Smilansky and Aljo{\v s}a Vol{\v c}i{\v c} for
their insightful remarks. I am grateful to the anonymous referee for constructive comments and suggestions. This work is part of the author's
PhD thesis and was supported by the Israel Science Foundation grant
number 2095/15 and by BSF grant 2016256.

\section{Examples}\label{sec: examples}

It is convenient to represent substitution schemes visually as a set of partitions of the elements of $\tau_\sigma$, where labels are visualized using colors, as is done throughout this section. Every such representation involves a choice of a substitution rule $\varrho_\sigma\in\Sigma_\sigma$, which in general may consist of various distinct substitution rules. A substitution scheme, as well as any sequence of partitions it generates, is said to be of \textit{constant configuration} if $\Sigma_\sigma$ consists of a single element. Although such examples are relevant to the study of hierarchical structures such as the tilings studied in \cite{SmiSol}, in this paper constant configuration is nowhere assumed. Nevertheless, unless otherwise stated, all illustrations of sequences of partitions below are of constant configuration, as these tend to be aesthetically pleasing. 

\begin{example}[$\alpha$-Kakutani]\label{ex:the alpha-Kakutani construction}
	The $\alpha$-Kakutani sequence of partitions can be represented as a Kakutani sequence of partitions generated by a multiscale substitution
	scheme of constant configuration in $\mathbb{R}^{1}$. This $\alpha$-Kakutani scheme is defined
	on a single prototile $\tau_\alpha=\left\{ \mathcal{I}\right\}$, with  $\omega_\alpha(\mathcal{I})=\left(\alpha\mathcal{I},(1-\alpha)\mathcal{I}\right)$ and $\Sigma_\alpha=\{\varrho_\alpha\}$,  where
	$\varrho_\alpha(\mathcal{I})=\alpha\mathcal{I}\sqcup\left(\alpha+(1-\alpha)\mathcal{I}\right)$.
	For example, the $\frac{1}{3}$-Kakutani scheme of constant configuration and substitution rule as illustrated in Figure \ref{fig: Kakutani 1/3-partition as multiscale substitution scheme},
	generates the $\frac{1}{3}$-Kakutani sequence described in Figure
	\ref{fig:a few elements in Kakutani third partition}. 
	
	\begin{figure}[H]
		\includegraphics[scale=1.4]{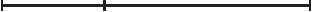}\caption{\label{fig: Kakutani 1/3-partition as multiscale substitution scheme}The
			$\frac{1}{3}$-Kakutani scheme on $\mathcal{I}$.}
	\end{figure}
	A $\frac{1}{3}$-Kakutani sequence of partitions of non-constant configuration generated by a similar scheme but with a second substitution rule $\varrho'_{1/3}(\mathcal{I})=\frac{2}{3}\mathcal{I}\sqcup\left(\frac{2}{3}+\frac{1}{3}\mathcal{I}\right)$ added to $\Sigma_{1/3}$, is shown in Figure \ref{fig: Kakutani kakutani partititions second version}.
	Note that indeed this sequence $\left\{ \pi_{m}\right\} $ is not of constant configuration,
	because when passing from $\pi_{0}$ to $\pi_{1}$ the substitution rule $\varrho'_{1/3}$ is applied, but when passing from $\pi_{1}$ to $\pi_{2}$ the
	substitution is done via $\varrho_{1/3}$. 
	\begin{figure}[H]
		\includegraphics[scale=1.56]{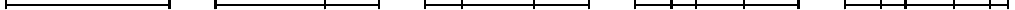}\caption{\label{fig: Kakutani kakutani partititions second version}A $\frac{1}{3}$-Kakutani
			sequence of partitions of non-constant configuration.}
	\end{figure}
	A generation sequence of partition generated by the $\frac{1}{3}$-Kakutani scheme is
	illustrated in Figure \ref{fig: Kakutani generation partition}.
	\begin{figure}[H]
		\includegraphics[scale=1.56]{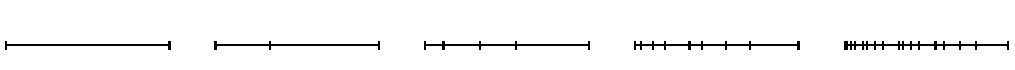}\caption{\label{fig: Kakutani generation partition}A generation sequence of
			partitions generated by the $\frac{1}{3}$-Kakutani scheme on $\mathcal{I}$.}
	\end{figure}
\end{example}

\begin{example}[Penrose-Robinson]
	\label{Penrose substitution example}The Penrose-type fixed scale
	substitution scheme, also known as the Penrose-Robinson substitution
	scheme, is described in Figure \ref{fig: Penrose fixed scale substitution scheme}.
	The prototiles are a tall triangle $\mathcal{T}$ and a short
	triangle $\mathcal{S}$ in $\mathbb{R}^{2}$, and the contraction
	constant is $\alpha=\frac{1}{\varphi}$, where $\varphi$ is the golden
	ratio. This substitution scheme is well known for generating aperiodic
	tilings of the Euclidean plane, and has been studied extensively,
	see Chapter 6 in \cite{BaakeGrimm} and references within.
	
	\begin{figure}[H]
		\includegraphics[scale=1.5]{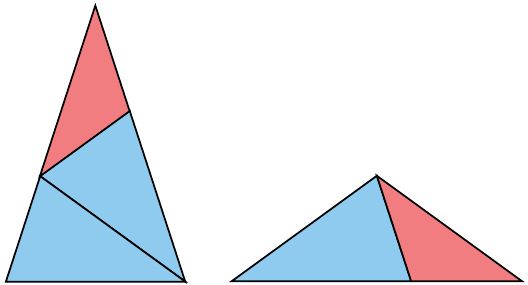}\caption{\label{fig: Penrose fixed scale substitution scheme}The Penrose-Robinson
			fixed scale substitution scheme.}
	\end{figure}
	Figure \ref{fig:Kakutani sequence of penrose} illustrates the first
	few partitions in a Kakutani sequence $\left\{ \pi_{m}\right\}$
	of $\mathcal{T}$. 
	\begin{figure}[H]
		\includegraphics[scale=1.5]{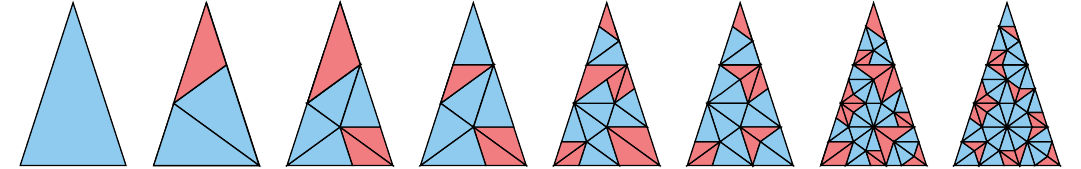}\caption{\label{fig:Kakutani sequence of penrose}A Kakutani sequence of partitions generated by the Penrose-Robinson scheme.}
	\end{figure}
	The first few partitions in a generation sequence of partitions $\left\{ \delta_{k}\right\} $
	of $\mathcal{T}$ are shown in
	Figure \ref{fig: penrose generation sequence}. 
	\begin{figure}[H]
		\includegraphics[scale=1.5]{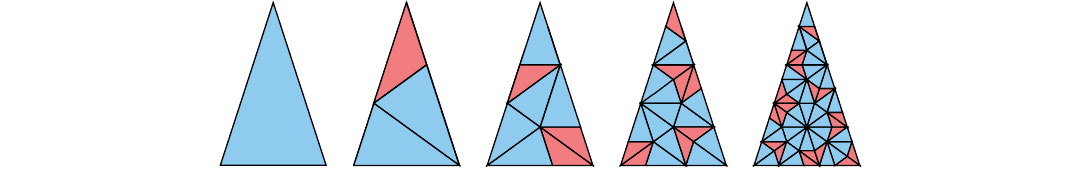}\caption{\label{fig: penrose generation sequence}A generation sequence of
			partitions generated by the Penrose-Robinson scheme.}
	\end{figure}
	
	By Theorem \ref{Thm: main result for generation sequences} the sequence
	$\left\{ \delta_{k}\right\}$ is uniformly distributed in $\mathcal{T}$. The two frequencies of types, as defined
	in Theorem \ref{thm:Types, u.d and frequencies} for the incommensurable
	case, can be also calculated using the formulas given in Theorem \ref{Thm: frequencies for generation fixed scale}
	\[
	\lim\limits _{k\rightarrow\infty}\frac{\left|\left\{ \text{Short triangles \ensuremath{\in\delta_{k}}}\right\} \right|}{\left|\left\{ \text{Tiles \ensuremath{\in\delta_{k}}}\right\} \right|}=\frac{1}{\varphi+1},\,\,\,\lim\limits _{k\rightarrow\infty}\frac{\vol\left(\bigcup\left\{ \text{Short triangles \ensuremath{\in\delta_{k}}}\right\} \right)}{\vol\left(\bigcup\left\{ \text{Tiles \ensuremath{\in\delta_{k}}}\right\} \right)}=\frac{1}{\varphi+2}.
	\]
	Further details are given in Example \ref{example:penrose statistics},
	where it is shown that in the case of the the Kakutani sequence $\left\{ \pi_{m}\right\} $
	illustrated in Figure \ref{fig:Kakutani sequence of penrose}, the corresponding limits do not exist.
\end{example}

\begin{rem}
	We note that the fact that the generation sequence
	of partitions $\left\{ \delta_{k}\right\}$ in Figure \ref{fig: penrose generation sequence} is a subsequence of the Kakutani sequence of partitions $\left\{ \pi_{m}\right\}$ in Figure \ref{fig:Kakutani sequence of penrose} is coincidental. In general
	this is not the case, as can be easily demonstrated by examples in which
	$\tau_\sigma$ contains two prototiles $ T_{i}$ and $T_{j}$
	for which $\vol\left(\alpha T_{i}\right)>\vol T_{j}$,
	where $\alpha$ is the contraction constant. 
\end{rem}

\begin{example}[The rectangle and the square]
	\label{example: multiscale substitution scheme on a rectangele and a square.}The
	construction shown in Figure \ref{fig: multiscale substitution scheme on rectangle and square}
	is an example of a multiscale substitution scheme $\sigma$ in $\mathbb{R}^{2}$ with two substitution rules $\Sigma_\sigma=\{\varrho_\sigma^{(1)},\varrho_\sigma^{(2)}\}$, where $\tau_\sigma=(\mathcal{R},\mathcal{S})$ and $\omega_\sigma$ are as illustrated in Figure \ref{fig: RS scheme}. 
	\begin{figure}[H]
		\includegraphics[scale=1.5]{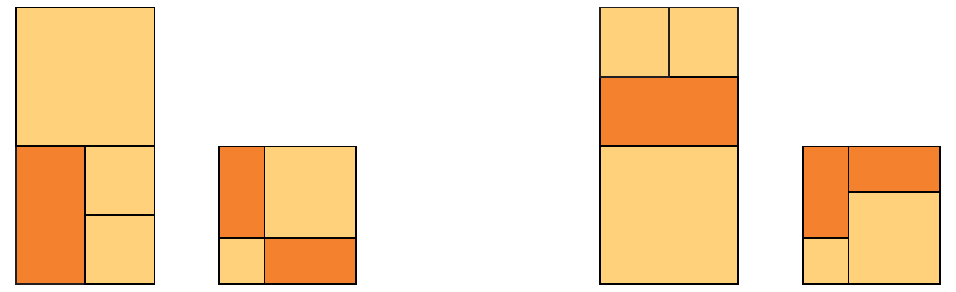}\caption{\label{fig: multiscale substitution scheme on rectangle and square}Two distinct substitution rules $\varrho_\sigma^{(1)}$ and $\varrho_\sigma^{(2)}$ in $\Sigma_\sigma$. }
	\end{figure}
	The substitution tiles are 
	$\omega_\sigma(\mathcal{R})=(\frac{1}{2}\mathcal{R},\mathcal{S},\frac{1}{2}\mathcal{S},\frac{1}{2}\mathcal{S})$ and 
	$\omega_\sigma(\mathcal{S})=(\frac{1}{3}\mathcal{R},\frac{1}{3}\mathcal{R},\frac{1}{3}\mathcal{S},\frac{2}{3}\mathcal{S}),
	$
	and the scheme is clearly incommensurable because $\frac{1}{3}\mathcal{S},\frac{2}{3}\mathcal{S}\in\omega_\sigma(\mathcal{S})$. The constants of substitution are
	\[
	\begin{array}{ll}
	\beta_{11}^{\left(1\right)}=\frac{1}{2} & \beta_{12}^{\left(1\right)}=\frac{1}{\sqrt{2}}, \beta_{12}^{\left(2\right)}=\beta_{12}^{\left(3\right)}=\frac{1}{2\sqrt{2}}\\
	\beta_{21}^{\left(1\right)}=\beta_{21}^{\left(2\right)}=\frac{\sqrt{2}}{3} & \beta_{22}^{\left(1\right)}=\frac{1}{3}, \beta_{22}^{\left(2\right)}=\frac{2}{3}
	\end{array}
	\]
	The first few elements
	of the constant configuration Kakutani sequence of partitions of the rectangle $\mathcal{R}$, with $\sigma=(\tau_\sigma,\omega_\sigma,\varrho_\sigma^{(1)})$ and $\varrho_\sigma^{(1)}$ as in the left hand side of Figure \ref{fig: multiscale substitution scheme on rectangle and square}, are illustrated in Figure \ref{fig: first element in sequence of partitions of square and rectangle}.
	\begin{figure}[H]
		\includegraphics[scale=1.5]{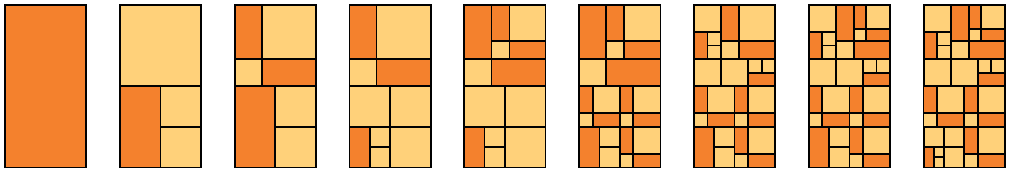}\caption{\label{fig: first element in sequence of partitions of square and rectangle}A
			Kakutani sequence of partitions of the rectangle $\mathcal{R}$.}
	\end{figure}
	
	The matrices $M_\sigma$ and $Q_\sigma$
	defined in Theorem \ref{thm:Types, u.d and frequencies} are given
	by
	\[
	M_\sigma=\left(\begin{array}{cc}
	\frac{1}{4} & \frac{3}{4}\\
	\frac{4}{9} & \frac{5}{9}
	\end{array}\right)\quad \quad
	Q_\sigma=\frac{\adj\left(I-M_\sigma\right)}{-\tr\left(\adj\left(I-M_\sigma\right)\cdot M_\sigma^{\prime}\right)}=\frac{1}{\frac{3}{4}\log3-\frac{1}{9}\log2}\left(\begin{array}{cc}
	\frac{4}{9} & \frac{3}{4}\\
	\frac{4}{9} & \frac{3}{4}
	\end{array}\right).
	\]
	In any Kakutani sequence of partitions generated by this scheme, either
	of the rectangle $\mathcal{R}$ or of the square $\mathcal{S}$, we get
	\[
	\lim\limits _{m\rightarrow\infty}\frac{\left|\left\{ \text{Squares \ensuremath{\in}\ensuremath{\ensuremath{\pi_{m}}}}\right\} \right|}{\left|\left\{ \text{Tiles \ensuremath{\in\pi_{m}}}\right\} \right|}=\frac{25}{43},\,\,\,\lim\limits _{m\rightarrow\infty}\frac{\vol\left(\bigcup\left\{ \text{Squares \ensuremath{\in}\ensuremath{\ensuremath{\pi_{m}}}}\right\} \right)}{\vol\left(\bigcup\left\{ \text{Tiles \ensuremath{\in\pi_{m}}}\right\} \right)}=\frac{\frac{5}{12}\log3-\frac{1}{18}\log2}{\frac{3}{4}\log3-\frac{1}{9}\log2}.
	\]
	Thus for large $m$ roughly $58\%$ of the tiles in $\pi_{m}$ are
	squares, covering approximately $56\%$ of the area.
\end{example}

\section{Preliminaries and reduction to counting}

\subsection{Constants of substitution and equivalence of multiscale substitution schemes}
\begin{lem}
	\label{lem: constants and volumes}Let $\sigma$ be a multiscale substitution
	scheme and let $\alpha_{ij}^{(k)}$ and $\beta_{ij}^{(k)}$
	be the constants appearing in Definition \ref{def: multicale substitution scheme}. Then
	\[
	\begin{array}{cc}
	\left(1\right) & \vol\left(\alpha_{ij}^{(k)} T_{j}\right)=\vol\left(\beta_{ij}^{(k)} T_{i}\right)\\
	\left(2\right) & \sum\limits _{j=1}^{n}\sum\limits _{k=1}^{k_{ij}}\left(\beta_{ij}^{(k)}\right)^{d}=1.
	\end{array}
	\]
\end{lem}

\begin{proof}
	This follows from Equation \eqref{eq: constants of substitution}
	and the fact that the volume of $ T_{i}$ is equal to the
	sum of volumes of the tiles in $ \omega_\sigma(T_i)$.
\end{proof}
\begin{defn}
	A multiscale substitution scheme $\sigma$ is called \textit{normalized} if all prototiles in $\tau_\sigma$ are of volume $1$. In a normalized scheme 
	\[
	\alpha_{ij}^{(k)}=\beta_{ij}^{(k)}
	\]
	for all $i,j$ and $k$. Two multiscale substitution schemes $\sigma_1$ and $\sigma_2$ are \textit{equivalent}
	if $\tau_{\sigma_{1}}$ and $\tau_{\sigma_{2}}$ consist of the same tiles
	up to scale changes, that is if 
	\[
	\tau_{\sigma_{1}}=\left( T_{1},\ldots, T_{n}\right)\Leftrightarrow\tau_{\sigma_{2}}=\left(\lambda_{1} T_{1},\ldots,\lambda_{n} T_{n}\right),
	\]
	and the tilings of the prototiles given by the substitution rules
	are the same up to a change of scales defined accordingly. Every equivalence
	class includes a single normalized scheme.
\end{defn}

For example, the non-normalized substitution scheme illustrated in Figure \ref{fig: RS scheme} is equivalent to the normalized one described in \ref{fig: Graph associated with MSS of rectangle and square}. Clearly the constants $\alpha_{ij}^{(k)}$ may vary between
equivalent schemes, as they depend on the volumes of the prototiles
and their ratios, but sequences of partition generated by equivalent
schemes are identical up to a uniform rescaling of all tiles in all
partitions. The next lemma follows from a straightforward computation
using Equation \eqref{eq: constants of substitution}.
\begin{lem}\label{lem: equivalent schemes have the same constants}
	Equivalent multiscale substitution schemes have the same constants
	of substitution $\beta_{ij}^{(k)}$. 
\end{lem}

The assumption that the constants of substitution $\beta_{ij}^{(k)}$
are a finite set of constants which are all strictly smaller than
$1$ yields the following statement.
\begin{lem}\label{lem: tiles shrink in sequences}
	Let $\sigma$ be a multiscale substitution scheme and let $\left\{ \gamma_{m}\right\} $ be either
	a Kakutani or a generation sequence of partitions of $ T_{i}\in\tau_\sigma$. Then for every $\varepsilon>0$ there exists $m_{0}\in\mathbb{N}$ such
	that all tiles in $\gamma_{m}$ are of diameter less than $\varepsilon$
	for all $m\geq m_{0}$.
\end{lem}

\subsection{Counting and uniform distribution}

A key step in the proof of Theorem \ref{Thm: Main result} is the
following Lemma. It implies that in order to prove uniform distribution
of sequences of partitions generated by multiscale substitution schemes,
it is enough to apply a counting argument.

\begin{defn}
	Let $\sigma$ be a multiscale substitution scheme and let $\{\gamma_m\}$ be a sequence of partitions of $T_i\in\tau_\sigma$ generated by $\sigma$. Denote by
	\[ T\in\mathscr{T}_i^{\sigma}(\{\gamma_m\})=\{T'\in\gamma_m\,:\,m\in\N\}
	\]
	the set of all tiles that appear in elements of the sequence $\{\gamma_m\}$. 
\end{defn}

Note that tiles of $\mathscr{T}_i^{\sigma}(\{\gamma_m\})$ are actual subsets of $T_i$, while those of $\mathscr{O}_i^\sigma$ from Definition \ref{def: all substitution tiles} are the substitution tiles copies of which appear in elements of $\{\gamma_m\}$. Clearly, if $\{\gamma_m\}$ is a Kakutani or a generation sequence then there is a one-to-one correspondence between $\mathscr{T}_i^{\sigma}(\{\gamma_m\})$ and  $\mathscr{O}_i^\sigma$. We write  $\mathscr{T}_i^{\sigma}=\mathscr{T}_i^{\sigma}(\{\gamma_m\})$ if the sequence of partitions $\{\gamma_m\}$ it refers to is clear from the context.  
 
\begin{lem}
	\label{Lemma: counting implies uniform distribution}Let $\sigma$ be a multiscale substitution scheme and let $\left\{ \gamma_{m}\right\}$
	be a sequence of partitions of $ T_{i}\in\tau_\sigma$ generated
	by $\sigma$. Assume that for
	every $\varepsilon>0$ there exists $m_{0}\in\mathbb{N}$ so that
	all tiles in $\gamma_{m}$ are of diameter
	less than $\varepsilon$ for any $m\geq m_{0}$. If there exists a marking sequence $\left\{ x_{m}\right\} $
	of $\left\{ \gamma_{m}\right\} $ so that 
	\[
	\lim_{m\rightarrow\infty}\frac{\left|x_{m}\cap T\right|}{\left|x_{m}\right|}=\frac{\vol T}{\vol T_{i}}
	\]
	holds for any tile $T\in\mathscr{T}_i^{\sigma}(\{\gamma_m\})$,  then $\{\gamma_m\}$ is uniformly distributed in $ T_{i}$.
\end{lem}

\begin{proof}
	Assume for simplicity $\vol T_{i}=1$. We want to show that
	\[
	\mu_{m}:=\frac{1}{\left|x_{m}\right|}\sum_{x\in x_{m}}\delta_{x}\longrightarrow\vol
	\]
	in the weak-{*} topology, where $\vol$ is the Lebesgue measure in
	$\mathbb{R}^{d}$, restricted to $ T_{i}$. 
	
	Let $\mu$ be a partial limit of $\mu_{m}$. It is enough to show
	that $\mu\left(R\right)=\vol R$ for any axis-parallel box $R\subset T_{i}$
	in order to prove $\mu=\vol$. By assumption, for any $T\in\mathscr{T}^\sigma_{i}$
	\[
	\mu(T)=\vol T.
	\]
	Let $\varepsilon>0$. Since $\mathscr{T}^\sigma_{i}$ contains tiles of arbitrarily
	small diameter, there exist subsets $I_{\varepsilon},C_{\varepsilon}\subset\mathscr{T}^\sigma_{i}$
	such that
	\[
	\bigsqcup_{T\in I_{\varepsilon}}T\subset R\subset\bigsqcup_{T\in C_{\varepsilon}}T
	\]
	and
	\[
	\vol R-\varepsilon<\sum_{T\in I_{\varepsilon}}\vol T\leq\vol R\leq\sum_{T\in C_{\varepsilon}}\vol T<\vol R+\varepsilon.
	\]
	It is enough to prove that for all $T\in\mathscr{T}^\sigma_{i}$
	\begin{equation}
	\mu\left(\partial T\right)=0,\label{eq: measure of boundary is zero}
	\end{equation}
	since then 
	\[
	\vol R-\varepsilon<\sum_{T\in I_{\varepsilon}}\vol T=\sum_{T\in I_{\varepsilon}}\mu(t)=\mu\left(\bigsqcup_{T\in I_{\varepsilon}}T\right)\leq\mu\left(R\right)
	\]
	and 
	\[
	\mu\left(R\right)\leq\mu\left(\bigsqcup_{T\in C_{\varepsilon}}T\right)=\sum_{T\in C_{\varepsilon}}\mu(t)=\sum_{T\in C_{\varepsilon}}\vol T<\vol R+\varepsilon.
	\]
	This holds for any $\varepsilon>0$ and the lemma follows.
	
	We now prove Equation \eqref{eq: measure of boundary is zero} for
	all $T\in\mathscr{T}^\sigma_{i}$. Let $T\in\mathscr{T}^\sigma_{i}$ and let $\eta>0$.
	Since $\vol\partial T=0$, there exists a countable sequence of cubes
	$Q_{n}$ such that
	\[
	\partial T\subset\bigcup Q_{n}
	\]
	and 
	\[
	\sum_{n}\vol Q_{n}<\eta.
	\]
	For any $\varepsilon_{n}>0$ there exists a covering of $Q_{n}$ by
	a countable set of tiles $T_{n}\subset\mathscr{T}^\sigma_{i}$ such that
	\[
	\mu\left(Q_{n}\right)\leq\mu\left(\bigcup_{T\in T_{n}}T\right)\leq\sum_{T\in T_{n}}\mu(t)=\sum_{T\in T_{n}}\vol T\leq\vol Q_{n}+\varepsilon_{n}.
	\]
	Choose $\varepsilon_{n}=\vol Q_{n}$. The boundary $\partial T$ is
	thus covered by a countable union of tiles in $\mathscr{T}^\sigma_{i}$,
	and 
	\[
	\mu\left(\partial T\right)\leq\mu\left(\bigcup_{n}\bigcup_{T\in T_{n}}T\right)\leq\sum_{n}\sum_{T\in T_{n}}\mu(t)\leq\sum_{n}\left(\vol Q_{n}+\varepsilon_{n}\right)=2\sum_{n}\vol Q_{n}<2\eta.
	\]
	This holds for any $\eta>0$, finishing the proof of Equation \eqref{eq: measure of boundary is zero}
	for all $T\in\mathscr{T}^\sigma_{i}$, thus proving the lemma.
\end{proof}
\begin{rem}
	A counterpart of Lemma \ref{Lemma: counting implies uniform distribution}
	appears as Lemma 2.5 in \cite{CarboneVolcic}.
\end{rem}

\section{\label{sec:Graphs}Graphs associated with multiscale substitution	schemes}

A key element in our study of Kakutani sequences of partitions is the directed weighted graph, which we regard as a geometric object, not only combinatorial. As in \cite{Graphs}, denote by $G=\left(\mathcal{V},\mathcal{E},l\right)$ a directed weighted metric multigraph with a set of vertices $\mathcal{V}$ and a set
of weighted edges $\mathcal{E}$, with positive weights which are
regarded as lengths. A \textit{path} in $G$ is a directed walk on
the edges of $G$ that originates and terminates at vertices of $G$.
More generally, a \textit{metric path} in $G$ is a directed walk on edges of
$G$, which does not necessarily originate or terminate at vertices
of $G$. An edge of weight $a$ is equipped with a parameterisation
by the interval $[0,a]$, and the parameterisation is used to define
the path metric $l$ on edges, paths and metric paths in $G$. We
assume throughout that as a subset of the metric graph, an edge contains
its terminal vertex but not its initial one.

\subsection{Graphs associated with multiscale substitution schemes}
\begin{defn}
	\label{def: associated graph}Let $\sigma$
	be a multiscale substitution scheme. The\textit{ graph} \textit{associated} with $\sigma$ is a directed weighted
	graph $G_\sigma$, the vertices of which model the prototiles in $\tau_\sigma$, and the weighted edges model the substitution tiles in $\omega_\sigma$ and their scales. More precisely, $G_\sigma$ has a set of vertices $\mathcal{V}_\sigma=\left\{ 1,\ldots,n\right\}$, where the prototile $T_{i}\in\tau_\sigma$ is associated with 
	the vertex $i\in\mathcal{V}_\sigma$,
	and every substitution tile in $ \omega_\sigma(T_i)$ is associated with a distinct edge in $\mathcal{E}_\sigma$  with initial vertex $i$. In addition, if
	$\varepsilon\in\mathcal{E}_\sigma$ is an edge in $G_\sigma$ associated
	with the substitution tile $\alpha T_{j}\in \omega_\sigma(T_i)$,
	then $\varepsilon$ terminates at vertex $j$
	and is of length 
	\[
	l\left(\varepsilon\right)=\log\frac{1}{\alpha}=\log\frac{1}{\beta}+\frac{1}{d}\left(\log\vol T_{j}-\log\vol T_{i}\right),
	\]
	where $\beta$ is the corresponding constant of substitution, as defined in Definition \ref{def: multicale substitution scheme}.
\end{defn}

Note that $G_\sigma$ does not depend on the substitution rules in $\Sigma_\sigma$, but only on the prototiles $\tau_\sigma$ and the substitution tiles $\omega_\sigma$. A priori, since in the definition of the multiscale substitution scheme we do not assume $\alpha<1$, under the definition above it may be the case that an edge in the associated graph carries a non-positive weight. For example, in the non-normalized scheme
defined in Figure \ref{fig: multiscale substitution scheme on rectangle and square}, the edge associated with $\mathcal{S}\in\omega_\sigma(\mathcal{R})$ 
is of ``length'' $0=\log1$. However, since we assume $0<\beta<1$ for all constants of substitution, all edges in graphs associated with normalized substitution schemes carry positive weights, and these are the graphs involved in the proof of the main results of this paper. 
\begin{example}
	The graph associated with the $\alpha$-Kakutani scheme is illustrated
	in Figure \ref{fig: Graph associated with kakutani}. It consists
	of a single vertex corresponding to the single prototile $\mathcal{I}$,
	and two loops of lengths $\log\frac{1}{\alpha}$ and $\log\frac{1}{1-\alpha}$.
	\begin{figure}[H]
		\includegraphics[scale=1.2
		]{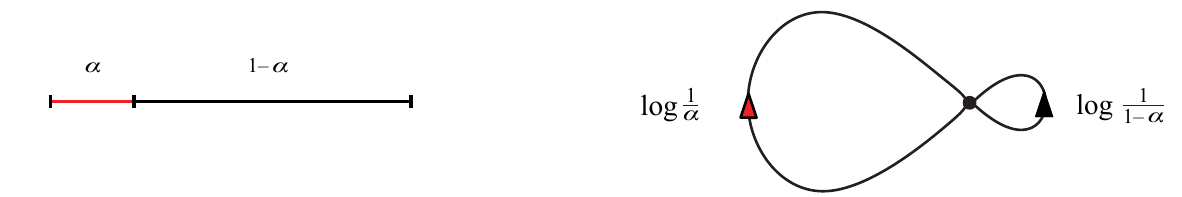}\caption{\label{fig: Graph associated with kakutani}The graph associated with
			the $\alpha$-Kakutani scheme on $\mathcal{I}$.}
	\end{figure}
	
\end{example}

\begin{example}
	\label{example: constants ang graph of the rectangle and square multiscale substitution scheme}The
	graph associated with a normalized multiscale substitution scheme,
	equivalent to the scheme represented in Figure \ref{fig: RS scheme},
	is illustrated in Figure \ref{fig: Graph associated with MSS of rectangle and square}.
	The left vertex is associated with the rectangle $\mathcal{R}$ and
	every one of its outgoing edges is associated with a distinct tile
	of matching color in the substitution rule on $\mathcal{R}$, where
	multiple arrow heads represent multiple distinct edges of the same
	length, origin, termination and direction. Similarly, the right vertex
	and its outgoing edges are associated with the square $\mathcal{S}$
	and the tiles which appear in its substitution rule. 
	\begin{figure}[H]
		\includegraphics[scale=1.2]{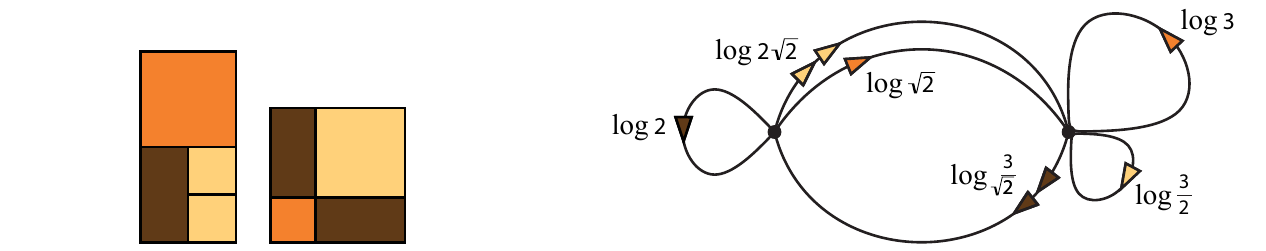}\caption{\label{fig: Graph associated with MSS of rectangle and square}The graph associated with a normalized scheme on $\mathcal{R}$ and $\mathcal{S}$. }
	\end{figure}
	
\end{example}

The next lemma describes the relation between associated graphs of
equivalent schemes, and follows from the definition of the associated
graph.
\begin{lem}
	\label{lem: Graphs and equivlence}Let $\sigma_1$ and $\sigma_2$ be two equivalent multiscale substitution schemes with $\tau_{\sigma_1}=\left( T_{1},\ldots, T_{n}\right)$,
	and $\tau_{\sigma_2}=\left( T_{1},\ldots,\alpha T_{j},\ldots, T_{n}\right)$,
	with $\alpha>0$, and let $G_{\sigma_1}$ and $G_{\sigma_2}$ be the associated graphs. 
	\begin{enumerate}
		\item The graphs $G_{\sigma_1}$ and $G_{\sigma_2}$ have the same sets of vertices and
		edges.
		\item The weights of the edges are not changed, except for edges $\varepsilon$
		with terminal vertex $j$, for which 
		\[
		l_{G_{\sigma_2}}\left(\varepsilon\right)=l_{G_{\sigma_1}}\left(\varepsilon\right)+\log\alpha,
		\]
		and edges $\varepsilon$ with initial vertex $j$, for which 
		\[
		l_{G_{\sigma_2}}\left(\varepsilon\right)=l_{G_{\sigma_1}}\left(\varepsilon\right)-\log\alpha.
		\]
	\end{enumerate}
\end{lem}

In view of Lemma \ref{lem: Graphs and equivlence}, in order to define the graph $G_{\sigma_2}$, one simply ``slides''
the vertex $j$ in $G_{\sigma_1}$ along the edges, in a way that does not change the length of any closed path in the associated graph. See also Example \ref{example: rhombus and triangle} and Figures \ref{fig:TR fixed scaled},
\ref{fig: TR normalized} and \ref{fig:TR rationalized} within.

\subsection{Paths in $G_\sigma$ and tiles in partitions}

Let $\sigma$
be a multiscale substitution scheme and let $G_\sigma$ be the associated graph. There is a natural
one-to-one correspondence between tiles of type $j$ in $\mathscr{O}^\sigma_{i}$
and finite paths in $G_\sigma$ with initial vertex $i\in\mathcal{V}_\sigma$ and
terminal vertex $j\in\mathcal{V}_\sigma$, where a tile of generation $m$
corresponds to a path consisting of $m$ edges. Denote by $\gamma_{T}$
the unique path in $G_\sigma$ that corresponds to the tile $T\in\mathscr{O}^\sigma_{i}$.

A partition of $ T_{i}$ that is an element either of a
Kakutani or of a generation sequence of partitions generated by the
scheme, corresponds to a finite collection of paths with initial vertex
$i\in\mathcal{V}_\sigma$. Every infinite path in $G_\sigma$ with initial vertex
$i$ is the continuation of a unique finite path in this finite collection. 
\begin{lem}
	\label{lemma: length of path and volume of tile}Let $T\in\mathscr{O}^\sigma_{i}$
	be a tile of type $j$, and assume $T$ corresponds to the path $\gamma_{T}$
	in $G_\sigma$. Then 
	\[
	\vol T=e^{-l\left(\gamma_{T}\right)d}\cdot\vol T_{j}.
	\]
\end{lem}

\begin{proof}
	This follows directly from the definition of $G_\sigma$. 
\end{proof}
\begin{cor}
	\label{cor: Graphs-associated-with equivalent schemes have the same closed paths}Graphs
	associated with equivalent multiscale substitution schemes have the
	same set of lengths of closed paths.
\end{cor}

\begin{proof}
	If $\gamma$ is a closed path in $G_\sigma$, that is if $i=j$, then $l\left(\gamma\right)$
	does not depend on the volumes of the tiles in $\tau_\sigma$, but
	only on the constants of substitution $\beta_{ij}^{(k)}$. The corollary now follows from Lemma \ref{lem: equivalent schemes have the same constants}
\end{proof}
\begin{cor}
	\label{cor: monotonicity of map between paths lengths and tile volumes}Let
	$G_\sigma$ be the graph associated with a normalized multiscale substitution
	scheme $\sigma$. Let $T_{1},T_{2}\in\mathscr{O}^\sigma_{i}$ and
	let $\gamma_{T_{1}}$ and $\gamma_{T_{2}}$ be the corresponding paths
	in $G_\sigma$. Then 
	\[
	\vol T_{1}<\vol T_{2}\,\,\,\,\,\text{{\rm \emph{if\,and\,only\,if}}}\,\,\,\,\,l\left(\gamma_{T_{1}}\right)>l\left(\gamma_{T_{2}}\right).
	\]
\end{cor}

\begin{proof}
	This is clear from Lemma \ref{lemma: length of path and volume of tile}.
\end{proof}
It follows that the set of lengths of paths in a graph associated
with a normalized scheme may be used as indices of a Kakutani sequences
of partitions, in the following sense: 
\begin{lem}
	\label{lem: tiles =00003D paths} Let $\sigma$
	be a multiscale substitution scheme, let $G_\sigma$ be the graph associated with an equivalent normalized scheme, and let $\{\pi_m\}$ be a Kakutani sequence of partitions of $T_i\in\tau_\sigma$. Let $\left\{ l_{m}\right\} $ be the increasing sequence of lengths of paths in $G_\sigma$ that originate at $i\in\mathcal{V}_\sigma$. Let $\varepsilon\in\mathcal{E}_\sigma$ be an edge associated with an element $\alpha T_{j}\in\omega_\sigma(T_h)$ for some $T_j,T_h\in\tau_\sigma$. There is a one-to-one correspondence between tiles of type $j$ in the partition $\pi_{m}$ that appear as a result of a substitution of a tile of type $h$ and are associated with the substitution tile $\alpha T_{j}\in\omega_\sigma(T_h)$,
	and metric paths of length $l_{m}$ that originate at $i\in\mathcal{V}_\sigma$ and terminate at a point on the edge $\varepsilon$. 
\end{lem}

\begin{proof}
	Let $T\in\pi_{m}$ be a tile associated with the substitution tile $\alpha T_{j}\in\omega_\sigma(T_h)$ as in the lemma. Then $T$ corresponds to a path $\gamma_{T}$ with initial vertex $i$, terminal vertex $j$ and final edge $\varepsilon$. If
	$T$ is of maximal volume in $\pi_{m}$, then $l\left(\gamma_{T}\right)=l_{m}$
	and $\gamma_{T}$ is also the metric path corresponding to $T$ as
	an element of $\pi_{m}$. Otherwise $l\left(\gamma_{T}\right)>l_{m}$,
	and the unique metric path of length exactly $l_{m}$ that is defined
	by truncating $\gamma_{T}$ terminates at a point on $\varepsilon$,
	and is the metric path corresponding to $T$ as an element of $\pi_{m}$. 
	
	Consider a metric path with initial vertex $i$ and length $l_{m}$
	that terminates at a point on $\varepsilon$, then it can be uniquely
	extended to a path $\gamma$ in $G_\sigma$ with final edge $\varepsilon$.
	Let $T\in\mathscr{T}^\sigma_{i}$ be the tile corresponding to $\gamma$.
	Then there exist $m_{0}\leq m\leq m_{1}$ such that $T$ first appears
	in partition $\pi_{m_{0}}$ and is of maximal volume among tiles in
	partition $\pi_{m_{1}}$. It follows that $T\in\pi_{m}$, and $T$
	is the tile corresponding to the metric path at hand.
\end{proof}

\begin{example}
	Consider the $\frac{1}{3}$-Kakutani multiscale substitution scheme and its associated graph, which consists of a single vertex and two loops, as illustrated in \ref{fig: Graph associated with kakutani}. The two loops are of lengths $\log 3$ and $\log \frac{3}{2}$, and so all paths are of lengths $a\log \frac{3}{2}+b\log 3$ for integers $a,b\ge0$. The first few elements of the sequence $\{l_m\}$ are therefore $l_0=0,l_1=\log \frac{3}{2}, l_2=2\log \frac{3}{2}$ and$l_3=\log 3$. Figure \ref{fig: Kakutani metric paths} illustrates all metric paths of lengths $l_0,l_1$ and $l_2$ together with their corresponding intervals in $\pi_0, \pi_1$ and $\pi_2$. As can be clearly seen form the illustration, the interval of maximal length  in the next partition $\pi_3$ would be the interval $[0,\frac{1}{3}]$, which corresponds to the loop of length $l_3=\log 3$. 
	
	\begin{figure}[H]
		\includegraphics[scale=1.5]{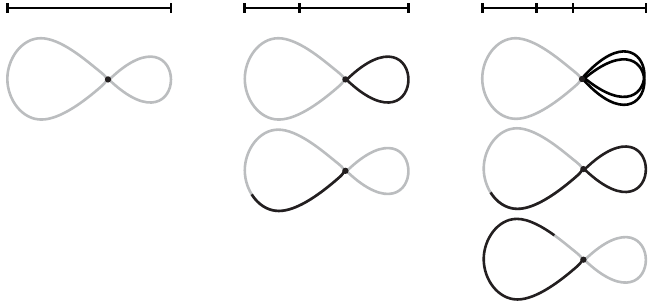}\caption{\label{fig: Kakutani metric paths}Elements of the $\frac{1}{3}$-Kakutani sequence and the corresponding metric paths.}
	\end{figure}

\end{example}

\subsection{The graph matrix function of graphs associated with multiscale substitutions}

\begin{defn}
	Let $G$ be a directed weighted graph with a set of vertices $\mathcal{V}=\left\{ 1,\ldots,n\right\} $.
	Let $i,j\in\mathcal{V}$ be a pair of vertices in $G$, and assume
	that there are $k_{ij}\geq0$ edges $\varepsilon_{1},\ldots,\varepsilon_{k_{ij}}$
	with initial vertex $i$ and terminal vertex $j$. The \textit{graph
		matrix function} of $G$ is the matrix valued function $M:\mathbb{C}\rightarrow M_{n}\left(\mathbb{C}\right)$
	defined by 
	\[
	M_{ij}(s)=e^{-s\cdot l\left(\varepsilon_{1}\right)}+\cdots+e^{-s\cdot l\left(\varepsilon_{k_{ij}}\right)}.
	\]
	If $i$ is not connected to $j$ by an edge, put $M_{ij}(s)=0$. Note that the restriction of $M$ to $\mathbb{R}$ is real valued.
\end{defn}

\begin{lem}
	\label{Lemma: M(d)  has eigenvalue 1 and eigenvector vol}Let $\sigma$
	be a multiscale substitution scheme in $\R^d$ and let $G_\sigma$ be the associated graph. Then $M(d)$ is a non-negative real
	valued matrix with a positive right eigenvector
	\[
	v_{\vol}=\left(\vol T_{1},\ldots,\vol T_{n}\right)^{T}\in\mathbb{R}^{n}
	\]
	and associated eigenvalue $\mu=1$.
\end{lem}

\begin{proof}
	The non-zero entries of the graph matrix function of $G_\sigma$ are given
	by 
	\[
	M_{ij}(s)=\sum_{k=1}^{k_{ij}}\left(\alpha_{ij}^{(k)}\right)^{s}.
	\]
	By Lemma \ref{lem: constants and volumes} we have 
	\[
	\begin{aligned}\left(M(d)\cdot v_{\text{vol}}\right)_{i} & =\sum_{j=1}^{n}M_{ij}(d)\vol T_{j}\\
	& =\sum_{j=1}^{n}\sum_{k=1}^{k_{ij}}\frac{\vol T_{i}}{\vol T_{j}}\left(\beta_{ij}^{(k)}\right)^{d}\vol T_{j}\\
	& =\vol T_{i},
	\end{aligned}
	\]
	and so $\text{\ensuremath{M(d)\cdot}}v_{\text{vol}}=1\cdot v_{\text{vol}}$.
\end{proof}

\subsection{Irreducibility of multiscale substitution schemes}
\begin{defn}
	A graph $G$ is called \textit{strongly connected} if for every pair
	of vertices $i,j\in\mathcal{V}$ there exists a path in $G$ with
	initial vertex $i$ and terminal vertex $j$. A non-negative real
	valued matrix $A\in M_{n}(\mathbb{R})$, that is, a matrix
	with non-negative entries, is called \textit{irreducible} if for every
	pair of indices $i,j$ there exists $k\in\mathbb{N}$ for which $\left(A^{k}\right)_{ij}>0$.
\end{defn}

Recall the definition of irreducible multiscale substitution schemes
given in Definition \ref{def:Irresucible multiscale substitution}.
\begin{lem}
	A multiscale substitution scheme is irreducible if and only if the
	associated graph $G_\sigma$ is strongly connected if and only if $M(d)$
	is an irreducible matrix.
\end{lem}

\begin{proof}
	This follows directly from the definitions of $G_\sigma$ and $M(d)$.
\end{proof}
The following result concerns irreducible matrices and is due to Perron
and Frobenius (full statements and proofs can be found in Chapter
XIII of \cite{Gantmacher}).
\begin{thm*}[Perron-Frobenius theorem for irreducible matrices]
	Let $A\in M_{n}(\mathbb{R})$ be a non-negative irreducible
	matrix. 
	\begin{enumerate}
		\item There exists $\mu>0$ which is a simple eigenvalue of $A$, and $\left|\mu_{j}\right|\leq\mu$
		for any other eigenvalue $\mu_{j}$. 
		\item There exists $v\in\mathbb{R}^{n}$ with positive entries such that
		$Av=\mu v$. Moreover every right eigenvector with non-negative entries
		is a positive multiple of $v$.
	\end{enumerate}
\end{thm*}
\begin{defn}
	The eigenvalue $\mu$ is called the \textit{Perron-Frobenius eigenvalue},
	and an associated positive right eigenvector is called a \textit{right
		Perron-Frobenius eigenvector}.
\end{defn}

\begin{cor}
	\label{cor: The Perron-Frobenius eigenvalue and eigenvector of M(d)}Let
	$M$ be the graph matrix function of a graph associated
	with an irreducible multiscale substitution scheme in $\mathbb{R}^{d}$.
	The Perron-Frobenius eigenvalue of $M(d)$ is $\mu=1$
	and $v_{\vol}$ is a right Perron-Frobenius eigenvector of $M(d)$.
\end{cor}

\begin{proof}
	This is a straightforward application of the Perron-Frobenius theorem
	for irreducible matrices, combined with Lemma \ref{Lemma: M(d)  has eigenvalue 1 and eigenvector vol}.
\end{proof}

\section{\label{sec:Incommensurable and commensurable and examples}Incommensurable and commensurable schemes}

\begin{defn}\label{def: incommensurable graphs}
	A strongly connected graph $G$ is \textit{incommensurable} if there exist two closed paths $\gamma_{1},\gamma_{2}$ in $G$ which are of
	incommensurable lengths, that is $\frac{l\left(\gamma_{1}\right)}{l\left(\gamma_{2}\right)}\notin\mathbb{Q}$. Otherwise, the graph is called \textit{commensurable}.
\end{defn}

\begin{lem}
	Let $\sigma$ be a multiscale substitution scheme. Then $\sigma$ is incommensurable if and only if the associated graph $G_\sigma$ is incommensurable. In addition, the incommensurability of a scheme depends only on its equivalence	class. 
\end{lem}

\begin{proof}
	By definition, there is a one-to-one correspondence between tiles in $\mathscr{O}_i^\sigma$ of type $i$, and closed paths in $G_\sigma$ with initial and terminal vertex $i\in\mathcal{V}_\sigma$. If $T$ is such a tile, then by Lemma \ref{lemma: length of path and volume of tile} the associated path $\gamma_T$ is of length 
	\[
	l(\gamma_T)=-\frac{1}{d}\log\frac{\vol T}{\vol T_j}.
	\]
	The equivalence is now clear from a comparison of Definitions \ref{def: incommensurable schemes} and \ref{def: incommensurable graphs} for incommensurable schemes and graphs, respectively. The second statement is immediate from Corollary \ref{cor: Graphs-associated-with equivalent schemes have the same closed paths}, which states that graphs associated with equivalent schemes have the same set of lengths of closed paths.
\end{proof}

\begin{lem}
	\label{lem: Incommensurable equivalent definitions}Let $G$ be a strongly connected directed weighted graph. Then G is incommensurable if and only if the set of lengths of all closed paths in $G$ is not a uniformly discrete subset of $\mathbb{R}$. 
\end{lem}

\begin{proof}
	The proof appears in the introduction of \cite{Graphs}, it is included here for the sake of completeness. First, assume  $G$ is incommensurable. By Dirichlet's
	approximation theorem, for every $\varepsilon>0$ there exist $p,q\in\mathbb{N}$
	such that $\left|l(\gamma_1)q-l(\gamma_2)p\right|<\varepsilon$, and so the set of lengths
	of closed paths in $G$ is not uniformly discrete. Conversely, if
	the set of lengths of closed paths is rationally dependent, then
	the finiteness of the graph implies that there is a finite set $L$
	of lengths for which the length of any closed path in $G$ is a linear
	combination with integer coefficients of elements in $L$. It follows
	that the set of lengths of closed paths in $G$ is uniformly discrete.
\end{proof}

\subsection{Some examples }
\begin{example}
	As we have seen, the Penrose-Robinson scheme illustrated in Figure \ref{fig: Penrose fixed scale substitution scheme}, as well as any other fixed scale scheme, is commensurable.  If the contracting
	constant is $\alpha$, then all edges of the associated
	graph are of equal length $\log\frac{1}{\alpha}$, and so the associated graph is clearly  commensurable.
\end{example}

\begin{example}
	The graph $G_\sigma$ associated with the multiscale substitution scheme on the rectangle $\mathcal{R}$
	and the square $\mathcal{S}$, both illustrated in Figure \ref{fig: Graph associated with MSS of rectangle and square}, is incommensurable. The graph
	$G_\sigma$ contains two loops of lengths $\log3$ and $\log2$ associated with the substitution tiles $\frac{1}{3}\mathcal{S},\frac{2}{3}\mathcal{S}\in\omega_\sigma(\mathcal{S})$, respectively. Since 
	$
	\frac{\log2}{\log3}\notin\mathbb{Q},
	$
	the associated graph $G_\sigma$ is incommensurable.
\end{example}

\begin{example}
	The Rauzy fractal introduced by Rauzy in \cite{Rauzy} is
	an example of a set with non-polygonal boundary that admits a multiscale
	substitution scheme, as illustrated in Figure \ref{fig: Rauzy fractal and its graph}.
	The substitution tiles are $\omega_\sigma(\mathfrak{R})=\left(\frac{1}{\tau}\mathfrak{R},\frac{1}{\tau^2}\mathfrak{R},\frac{1}{\tau^3}\mathfrak{R}\right)$
	where $\tau$ is the tribonacci constant satisfying $\tau^3+\tau^2+\tau=1$. The associated graph has
	a single vertex corresponding to the single prototile $\mathfrak{R}$,
	and three loops of lengths $\log\tau,2\log\tau$ and $3\log\tau$,
	and so the scheme is not fixed scale but is nevertheless commensurable.
	\begin{figure}[H]
		\includegraphics[scale=0.2]{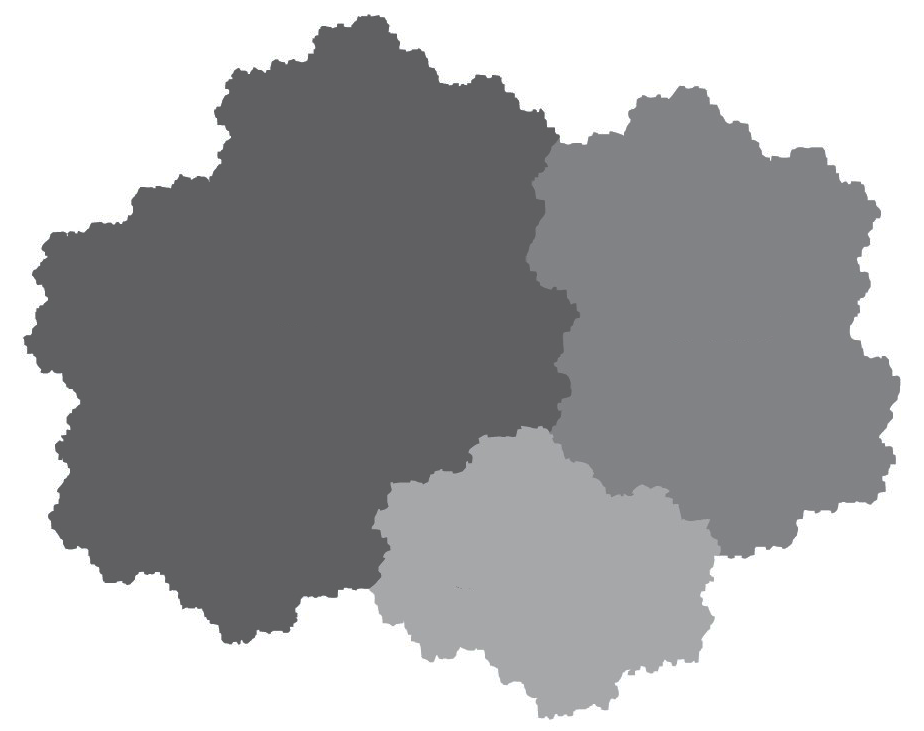}\caption{\label{fig: Rauzy fractal and its graph}The commensurable scheme
			on the Rauzy fractal $\mathfrak{R}$, courtesy of \cite{Wiki}. }
	\end{figure}
\end{example}

\begin{example}
	The $\alpha$-Kakutani scheme on the unit interval $\mathcal{I}$ is incommensurable
	for all but countably many values of $\alpha\in\left(0,1\right)$.
	As illustrated in Figure \ref{fig: Graph associated with kakutani},
	the associated graph consists of a single vertex and two loops of
	lengths $\log\frac{1}{\alpha}$ and $\log\frac{1}{1-\alpha}$, and
	so the $\alpha$-Kakutani scheme is commensurable if and only if $\frac{\log\alpha}{\log\left(1-\alpha\right)}\in\mathbb{Q}$.
	Indeed as we have seen, the $\frac{1}{3}$-Kakutani scheme is incommensurable,
	while the $\frac{1}{\varphi}$-Kakutani scheme is commensurable, where
	$\varphi$ is the golden ratio. 
\end{example}

\begin{example}
	The generalized pinwheel construction introduced by Sadun and used
	in \cite{Sadun} to define and study tilings of the Euclidean
	plane, can be viewed as a one parameter family of multiscale substitution
	schemes. Each scheme is defined on a single prototile which is a right
	triangle with an angle $\theta$, and the substitution rule defines
	a decomposition of the triangle into five rescaled and rotated similar
	triangles, four of which of the same volume, as illustrated in Figure
	\ref{fig: Generalized pinwheel}. The generalized pinwheel scheme
	is incommensurable for all but countably many values of $\theta\in\left(0,\frac{\pi}{2}\right)$.
	Indeed, the two constants of substitution depend on $\theta$ and
	the scheme is commensurable if and only if $\frac{\log\sin\theta}{\log\frac{\cos\theta}{2}}\in\mathbb{Q}$. 
	
	\begin{figure}[H]
		\includegraphics[scale=1.5]{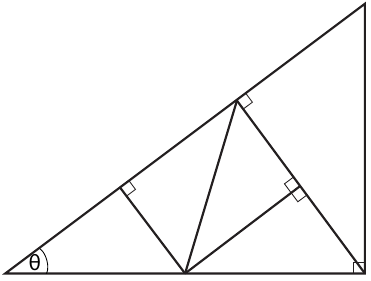}\caption{\label{fig: Generalized pinwheel}The generalized pinwheel scheme.}
	\end{figure}
\end{example}

\begin{rem}
	The tilings of the plane studied in \cite{Sadun} are generated
	by a succession of applications of the substitution rule on triangles,
	each followed by a suitable rescaling and repositioning. The substitution
	rule is always applied to triangles of maximal volume, as is done
	in the Kakutani splitting procedure. This defines a sequence of nested
	partitions of finite regions, the union of which is a tiling of the entire
	plane. The fixed scale scheme defined by putting $\theta=\arctan\frac{1}{2}$
	generates the Conway and Radin's classical pinwheel tiling, see \cite{Radin}.
	Measures associated with sequences of partitions generated by this
	construction are studied in \cite{Olli}, where uniform
	distribution for the corresponding Kakutani sequences is established. 
\end{rem}

\section{\label{sec: sequences generated by fixed scaled}Generation sequences generated by fixed scale substitution schemes}

\subsection{The substitution matrix and the associated graph}
\begin{defn}
	Let $\sigma$ be a fixed scale scheme in $\mathbb{R}^{d}$ with contraction constant $\alpha$. The \textit{substitution matrix} of $\sigma$ is an integer
	valued matrix $S_\sigma$ with entries 
	\[
	(S_\sigma)_{ij}=k_{ij}=\#\left\{ \text{copies of \ensuremath{\alpha T_{j}} in \ensuremath{ \omega_\sigma(T_i)}}\right\} .
	\]
\end{defn}

All edges of the graph $G_\sigma$ associated with a fixed scale substitution
scheme with contraction constant $\alpha$ are of length $\log\frac{1}{\alpha}$.
The entries of the graph matrix function as defined in Section \ref{sec:Graphs}
are thus given by 
\[
M_{ij}(s)=k_{ij}\alpha^{s},
\]
and so $S_\sigma$, which is also the adjacency matrix of $G_\sigma$, can be written
as

\[
S_\sigma=\frac{1}{\alpha^{d}}M(d).
\]
By Lemma \ref{Lemma: M(d)  has eigenvalue 1 and eigenvector vol},
$\mu=1$ is a Perron-Frobenius eigenvalue and $v_{\text{vol}}$ is
a right Perron-Frobenius eigenvector of $M(d)$, and so
\[
S_\sigma\cdot v_{\vol}=\frac{1}{\alpha^{d}}\cdot v_{\vol}.
\]
Therefore $\mu=\frac{1}{\alpha^{d}}$ is a Perron-Frobenius eigenvalue
and $v_{\text{vol}}$ is a right Perron-Frobenius eigenvector of $S_\sigma$.
\begin{example}
	The substitution matrix of the Penrose-Robinson fixed scale substitution
	scheme on a tall triangle $\mathcal{T}$ and a short triangle $\mathcal{S}$
	in $\mathbb{R}^{2}$, as described in Figure \ref{fig: Penrose fixed scale substitution scheme},
	is given by
	\[
	S_\sigma=\left(\begin{array}{cc}
	2 & 1\\
	1 & 1
	\end{array}\right).
	\]
	The Perron-Frobenius eigenvalue of $S_\sigma$ is $\mu=\varphi^{2}$, and
	a right Perron-Frobenius eigenvector can be chosen to be $v=\left(\varphi,1\right)$, where $\varphi$ is the golden ratio.
	Indeed, the contracting constant of the scheme is $\alpha=\frac{1}{\varphi}$,
	and the scale of the prototiles can be chosen so that $\vol \mathcal{T}=\varphi$
	and $\vol\mathcal{S}=1$, and so $v=v_{\vol}$. 
\end{example}

\begin{lem}
	\label{lem: tiles in generations}Let $\sigma$ be a fixed scale scheme, and let $\left\{ \delta_{k}\right\} $ be a
	generation sequence of partitions of $T_{i}\in\tau_\sigma$. Then 
	\[
	\left(S_\sigma^{k}\right)_{ij}=\left|\left\{ \text{{\rm Tiles of type \ensuremath{j} and generation \ensuremath{k} in a tiling of \ensuremath{ T_{i}}}}\right\} \right|.
	\]
\end{lem}

\begin{proof}
	This is a standard result about adjacency matrices of graphs, and
	is proved by induction on $k$. 
\end{proof}

\subsection{Primitive fixed scale substitution schemes}
\begin{defn}
	An irreducible real valued matrix $A\in M_{n}(\mathbb{R})$
	is called\textit{ primitive} if there exists $k\in\mathbb{N}$ for
	which $A^{k}$ is a matrix with all entries strictly positive, and\textit{
		non-primitive} otherwise. An irreducible fixed scale substitution scheme
	is called \textit{primitive} if the associated substitution matrix
	is primitive, and \textit{non-primitive} otherwise.
\end{defn}

\begin{rem}
	In the case of a fixed scale substitution scheme $\sigma$ with contacting constant
	$\alpha$, the substitution matrix $S_\sigma$ is primitive if and only if
	$G_\sigma$ is \textit{aperiodic}, which in this case means that the greatest
	common divisor of the set of lengths of all closed paths is $\log\frac{1}{\alpha}$. 
\end{rem}

The following is the Perron-Frobenius theorem for primitive matrices,
which gives a stronger result than its counterpart for general irreducible
matrices (see Chapter XIII of \cite{Gantmacher} for more details).
\begin{thm*}[Perron-Frobenius theorem for primitive matrices]
	Let $A\in M_{n}(\mathbb{R})$ be a primitive matrix.
	\begin{enumerate}
		\item There exists $\mu>0$ which is a simple eigenvalue of $A$, and $\left|\mu_{j}\right|<\mu$
		for any other eigenvalue $\mu_{j}$. 
		\item There exist $v,u\in\mathbb{R}^{n}$ with positive entries such that
		$Av=\mu v$ and $u^{T}A=\mu u^{T}$. Moreover, every right eigenvector
		with non-negative entries is a positive multiple of $v$, and every
		left eigenvector with non-negative entries is a positive multiple
		of $u$.
		\item The following holds
		\[
		\lim_{k\rightarrow\infty}\left(\frac{1}{\mu}A\right)^{k}=\frac{vu^{T}}{u^{T}v}.
		\]
	\end{enumerate}
\end{thm*}
\begin{defn}
	The limit matrix $P=\frac{vu^{T}}{u^{T}v}\in M_{n}(\mathbb{R})$
	is called the \textit{Perron projection} of $A$.
\end{defn}

\subsection{Irreducible non-primitive fixed scale substitution schemes}
\begin{defn}
	A matrix $A\in M_{n}(\mathbb{R})$ is a \textit{cyclic
		matrix of period} $p$ if up to a permutation of the indices, $A$
	is of the block form 
	\[
	A=\begin{pmatrix}0_{0} & A_{0}\\
	& 0_{1} & \ddots\\
	&  & \ddots & A_{p-2}\\
	A_{p-1} &  &  & 0_{p-1}
	\end{pmatrix},
	\]
	where $n_{0}+\cdots+n_{p-1}=n$, and $0_{r}\in M_{n_{r}}(\mathbb{R})$
	is the square zero matrix of order $n_{r}$.
\end{defn}

Proofs of the following result can be found in the discussion on irreducible
matrices in Part 1 of \cite{Seneta}.
\begin{lem}
	\label{lem:non-primitive matrix structure}Let $A\in M_{n}(\mathbb{R})$
	be an irreducible non-primitive matrix with Perron-Frobenius eigenvalue
	$\mu$. Then there exists $p\geq2$ such that $A$ is a cyclic matrix
	of period $p$. In addition 
	\[
	A^{p}=\begin{pmatrix}B_{0}\\
	& \ddots\\
	&  & B_{p-1}
	\end{pmatrix},
	\]
	where $B_{r}\in M_{n_{r}}(\mathbb{R})$ are primitive square
	matrices with Perron-Frobenius eigenvalue $\mu^{p}$.
\end{lem}

\begin{cor}
	Let $\sigma$ be an irreducible non-primitive fixed scale substitution scheme.
	Then the prototiles in $\tau_\sigma$ can be divided into $p$ \textit{classes} $\mathcal{C}_{0},\ldots,\mathcal{C}_{p-1}$,
	so that the substitution tiles of prototile in $\mathcal{C}_{r}$
	contain only tiles of types in $\mathcal{C}_{r+1\left(\mod p\right)}$. 
\end{cor}

In addition, the associated generation sequences of partition have
the following useful property.
\begin{cor}
	\label{cor: generation of non-primitive is a union of generation}Let
	$\sigma$
	be an irreducible non-primitive scheme on, and let $\left\{ \delta_{k}\right\} $ be a
	generation sequence of partitions of $ T_{i}\in\tau_\sigma$.
	Assume without loss of generality that $ T_{i}\in\mathcal{C}_{0}$. 
	\begin{enumerate}
		\item The sequence $\left\{ \delta_{k}\right\} $ consists of $p$ subsequences
		$\left\{ \delta_{pk+r}\right\} $, and partition $\delta_{pk+r}$
		contains only tiles of class $\mathcal{C}_{r}$.
		\item The subsequence $\left\{ \delta_{pk+r}\right\} $ is geometrically
		identical to the union of generation sequences of partitions of tiles
		which appear in the partition $\delta_{r}$, generated by an irreducible
		primitive fixed scale substitution scheme on $\mathcal{C}_{r}$. The
		substitution matrix of this scheme is the primitive matrix $B_{r}$
		described in Lemma \ref{lem:non-primitive matrix structure}, and
		the contraction constant is $\alpha^{p}$.
	\end{enumerate}
\end{cor}

\begin{example}
	The substitution defined in Figure \ref{fig:non-primitive} is an
	irreducible non-primitive fixed scale substitution scheme on a square
	and two rectangles, with period $p=2$ and contraction constant $\alpha=\frac{1}{\sqrt{3}}$.
	\begin{figure}[H]
		\includegraphics[scale=0.8]{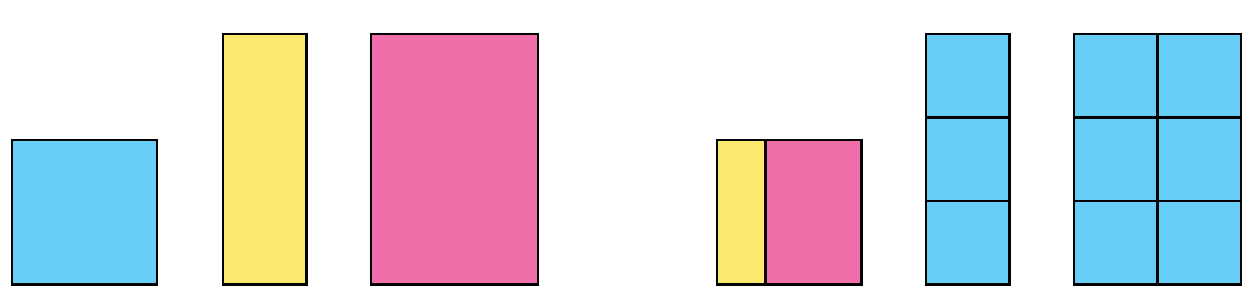}\caption{\label{fig:non-primitive}A non-primitive fixed scale substitution
			scheme with period $p=2$.}
	\end{figure}
	The first class of prototiles $\mathcal{C}_{0}$ consists of the square,
	and the second class $\mathcal{C}_{1}$ consists of the two rectangles,
	therefore $n_{0}=1$ and $n_{1}=2$. The substitution matrix $S_\sigma$
	and the block diagonal $S_\sigma^{2}$, with two primitive blocks $B_{0}$
	and $B_{1}$, are given by
	\[
	S_\sigma=\left(\begin{array}{ccc}
	0 & 1 & 1\\
	3 & 0 & 0\\
	6 & 0 & 0
	\end{array}\right),\,\,\,S_\sigma^{2}=\left(\begin{array}{ccc}
	9 & 0 & 0\\
	0 & 3 & 3\\
	0 & 6 & 6
	\end{array}\right).
	\]
\end{example}

\subsection{Proof of uniform distribution of generation sequences generated by
	fixed scale schemes}
\begin{proof}[Proof of Theorem \ref{Thm: main result for generation sequences}]
	Let $T\in\mathscr{T}^\sigma_{i}$ and assume that $T$ is a tile of type $j$
	that appears at partition $\delta_{k_{0}}$.   
	
	First, assume that scheme $\sigma$ is primitive. By Lemma \ref{lem: tiles in generations},
	for any $k>k_{0}$ we have
	\[
	\frac{\left|x_{k}\cap T\right|}{\left|x_{k}\right|}=\frac{\sum_{h=1}^{n}\left(S_\sigma^{k-k_{0}}\right)_{jh}}{\sum_{h=1}^{n}\left(S_\sigma^{k}\right)_{jh}}=\frac{\left(S_\sigma^{k-k_{0}}\cdot\boldsymbol{1}\right)_{j}}{\left(S_\sigma^{k}\cdot\boldsymbol{1}\right)_{i}},
	\]
	where $\boldsymbol{1}=\left(1,\ldots,1\right)^{T}\in\mathbb{R}^{n}$.
	The Perron-Frobenius eigenvalue of $S_\sigma$ is $\mu=\frac{1}{\alpha^{d}}$
	and $v_{\vol}$ is a Perron-Frobenius eigenvector of $S_\sigma$. The matrix
	$S_\sigma$ is primitive, the Perron projection of $S_\sigma$ satisfies
	\[
	\lim_{k\rightarrow\infty}\left(\frac{1}{\mu}S_\sigma\right)^{k}=\lim_{k\rightarrow\infty}\left(\alpha^{d}S_\sigma\right)^{k}=P,
	\]
	and the columns of $P$ are spanned by $v_{\vol}$. From this we get
	\[
	\begin{aligned}\lim_{k\rightarrow\infty}\frac{\left|x_{k}\cap T\right|}{\left|x_{k}\right|} & =\lim_{k\rightarrow\infty}\frac{\left(S_\sigma^{k-k_{0}}\cdot\boldsymbol{1}\right)_{j}}{\left(S_\sigma^{k}\cdot\boldsymbol{1}\right)_{i}}=\lim_{k\rightarrow\infty}\left(\alpha^{d}\right)^{k_{0}}\cdot\frac{\left(\alpha^{d}\right)^{k-k_{0}}\left(S_\sigma^{k-k_{0}}\cdot\boldsymbol{1}\right)_{j}}{\left(\alpha^{d}\right)^{k}\left(S_\sigma^{k}\cdot\boldsymbol{1}\right)_{i}}\\
	& =\left(\alpha^{d}\right)^{k_{0}}\lim_{k\rightarrow\infty}\frac{\left(\left(\alpha^{d}S_\sigma\right)^{k-k_{0}}\cdot\boldsymbol{1}\right)_{j}}{\left(\left(\alpha^{d}S_\sigma\right)^{k}\cdot\boldsymbol{1}\right)_{i}}=\left(\alpha^{d}\right)^{k_{0}}\cdot\frac{\left(P\cdot\boldsymbol{1}\right)_{j}}{\left(P\cdot\boldsymbol{1}\right)_{i}}\\
	& =\left(\alpha^{d}\right)^{k_{0}}\cdot\frac{\vol T_{j}}{\vol T_{i}}=\frac{\vol\left(\alpha^{k_{0}}T_{j}\right)}{\vol T_{i}}=\frac{\vol T}{\vol T_{i}}.
	\end{aligned}
	\]
	By Lemma \ref{Lemma: counting implies uniform distribution} uniform
	distribution of $\left\{ \delta_{k}\right\} $ is established for
	the primitive case. 
	
	Next, assume the scheme is non-primitive with period $p$, and assume
	without loss of generality that $ T_{i}\in\mathcal{C}_{0}$.
	Let $T\in\mathscr{T}^\sigma_{i}$ and assume that $T\in\delta_{k_{0}}$ and
	$k_{0}\equiv r_{0}\,\,\mod p$, so $T$ is a rescaled copy of a tile
	in $\mathcal{C}_{r_{0}}$. For every $0\leq c\leq p-1$, let $T_{1}^{\left(c\right)},\ldots,T_{s_{c}}^{\left(c\right)}$
	be the tiles in the partition $\delta_{k_{0}+c}$ such that 
	\[
	T=\bigsqcup_{i=1}^{s_{c}}T_{i}^{\left(c\right)}.
	\]
	Let $0\leq r\leq p-1$ and let $c\equiv r-r_{0}\,\,\mod p$. By Corollary
	\ref{cor: generation of non-primitive is a union of generation} and
	what we have just proved for primitive schemes 
	\[
	\lim_{k\rightarrow\infty}\frac{\left|x_{pk+r}\cap T\right|}{\left|x_{pk+r}\right|}=\sum_{i=1}^{s_{c}}\lim_{k\rightarrow\infty}\frac{\left|x_{pk+r}\cap T_{i}\right|}{\left|x_{pk+r}\right|}=\sum_{i=1}^{s_{c}}\frac{\vol T_{i}}{\vol T_{i}}=\frac{\vol T}{\vol T_{i}}.
	\]
	The limit does not depend on $r$, and so 
	\[
	\lim_{k\rightarrow\infty}\frac{\left|x_{k}\cap T\right|}{\left|x_{k}\right|}=\frac{\vol T}{\vol T_{i}}.
	\]
	By Lemma \ref{Lemma: counting implies uniform distribution} this finishes the proof of Theorem \ref{Thm: main result for generation sequences}.
\end{proof}

\subsection{Type frequencies in generation sequences generated by primitive fixed
	scale schemes}

The following is the counterpart of Theorem \ref{thm:Types, u.d and frequencies}
for generation sequences generated by primitive fixed scale substitution
schemes. These results follow from the theory of Perron-Frobenius, and are included
here to allow comparison with their incommensurable counterparts.
\begin{thm}
	\label{Thm: frequencies for generation fixed scale}Let $\sigma$ be an irreducible primitive fixed scale scheme in $\mathbb{R}^{d}$ with contraction constant $\alpha$. Let $\left\{ \delta_{k}\right\} $ be a
	generation sequence of partitions of $ T_{i}\in\tau_\sigma$, and let $1\leq r\leq n$. 
	\begin{enumerate}
		\item Any $r$-marking sequence $\{ x_{k}^{\left(r\right)}\} $
		of $\left\{ \delta_{k}\right\} $ is uniformly distributed in $ T_{i}$. 
		\item The ratio between the number of tiles of type $r$ in $\delta_{k}$
		and the total number of tiles is
		\[
		u_{r}+o\left(1\right),\quad k\rightarrow\infty,
		\]
		where $u=\left(u_{1},\ldots,u_{n}\right)^{T}$ is a left Perron-Frobenius
		eigenvector of $S$ normalized so that $\sum u_{j}=1$.
		\item The volume of the region covered by tiles of type $r$ in $\delta_{k}$
		is 
		\[
		\vol\left( T_{i}\right)w_{r}+o\left(1\right),\quad k\rightarrow\infty,
		\]
		where $w=\left(w_{1},\ldots,w_{n}\right)^{T}$ is a left Perron-Frobenius
		eigenvector of $M(d)$ normalized such that $\sum w_{j}=1$,
		and $M$ is the graph matrix function of the graph associated
		with the equivalent normalized scheme.
	\end{enumerate}
\end{thm}

\begin{rem}
	Estimates for the error terms can be found in the literature, see
	for example \cite{Yaar}.
\end{rem}

\begin{proof}
	The uniform distribution of any $r$-marking sequence $\{ x_{k}^{\left(r\right)}\} $
	of $\left\{ \delta_{k}\right\} $ follows from the same arguments
	as those described above, by replacing the vector $\boldsymbol{1}\in\mathbb{R}^{n}$
	with $e_{r}\in\mathbb{R}^{n}$, the vector $r$ of the standard
	basis of $\mathbb{R}^{n}$. 
	
	The second part of the theorem follows from the Perron-Frobenius theorem
	applied to the substitution matrix $S_\sigma$ and a direct calculation:
	\[
	\lim_{k\rightarrow\infty}\frac{|x_{k}^{\left(r\right)}|}{\left|x_{k}\right|}=\lim_{k\rightarrow\infty}\frac{\left(S_\sigma^{k}\cdot e_{r}\right)_{i}}{\left(S_\sigma^{k}\cdot\boldsymbol{1}\right)_{i}}=\frac{\left(vu^{T}\cdot e_{r}\right)_{i}}{\left(vu^{T}\cdot\boldsymbol{1}\right)_{i}}=\frac{u_{r}v_{i}}{v_{i}}=u_{r}.
	\]
	
	To prove the third part of the theorem, define the \textit{weighted substitution matrix} by 
	\[
	(W_\sigma)_{ij}=\frac{\vol\left(\alpha T_{j}\right)}{\vol T_{i}}k_{ij}.
	\]
	Similarly to Lemma \ref{lem: tiles in generations}, the matrix $W_\sigma$ satisfies
	\[
	\left(W_\sigma^{k}\right)_{ir}=\frac{\vol\left(\bigcup\left\{ \text{Tiles of type \ensuremath{r} \ensuremath{\in}\ensuremath{\ensuremath{\delta_{k}}}}\right\} \right)}{\vol T_{i}},
	\]
	and is primitive if and only if $S_\sigma$ is primitive. Observe that $\frac{\vol \left(\alpha T_{j}\right)}{\vol T_{i}}=\beta^{(k)}_{ij}$ for all $k=1,\ldots,k_{ij}$, and so $W_\sigma=M(d)$,
	where $M$ is the graph matrix function of the graph
	associated with the equivalent normalized scheme. It follows by Lemma \ref{Lemma: M(d)  has eigenvalue 1 and eigenvector vol} that
	$W_\sigma$ has Perron-Frobenius eigenvalue $\mu=1$ and a right positive
	Perron-Frobenius eigenvector $v=\boldsymbol{1}\in\mathbb{R}^{n}$.
	
	Let $w\in\mathbb{R}^{n}$ be a left positive eigenvector of $W_\sigma$ chosen
	such that $\sum w_{j}=1$, then
	\[
	\lim_{k\rightarrow\infty}\frac{\vol\left(\bigcup\left\{ \text{Tiles of type \ensuremath{r} \ensuremath{\in}\ensuremath{\ensuremath{\delta_{k}}}}\right\} \right)}{\vol T_{i}}=\lim_{k\rightarrow\infty}\left(W_\sigma^{k}\cdot e_{r}\right)_{i}=\left(\frac{vw^{T}}{w^{T}v}\cdot e_{r}\right)_{i}=\left(\frac{\boldsymbol{1}w^{T}}{w^{T}\boldsymbol{1}}\cdot e_{r}\right)_{i}=w_{r}
	\]
	finishing the proof. 
\end{proof}
\begin{example}
	\label{example:penrose statistics}Consider the Penrose-Robinson fixed
	scale substitution scheme on a tall triangle $\mathcal{T}$ and a
	short triangle $\mathcal{S}$, as defined in Figure \ref{fig: Penrose fixed scale substitution scheme}.
	The substitution matrix and the weighted substitution matrices are
	given by 
	\[
	S_\sigma=\left(\begin{array}{cc}
	2 & 1\\
	1 & 1
	\end{array}\right),\,\,\,W_\sigma=\frac{1}{\varphi^{2}}\left(\begin{array}{cc}
	2 & \varphi^{-1}\\
	\varphi & 1
	\end{array}\right),
	\]
	and so $u^{T}=\left(\frac{\varphi}{\varphi+1},\frac{1}{\varphi+1}\right)$
	and $w^{T}=\left(\frac{\varphi+1}{\varphi+2},\frac{1}{\varphi+2}\right)$,
	from which the frequencies of the generation sequence of partitions
	$\left\{ \delta_{k}\right\} $ of $\mathcal{T}$ illustrated in Figure
	\ref{fig: penrose generation sequence} are deduced, see Example \ref{Penrose substitution example}.
	
	Consider now the Kakutani sequence of partitions $\left\{ \pi_{m}\right\} $
	of $ T$ generated by the Penrose-Robinson scheme, illustrated
	in Figure \ref{fig:Kakutani sequence of penrose}. Here $\delta_{k}=\pi_{2k-1}$
	for all $k\in\mathbb{N}$, and so the number of tall and short triangles
	in $\pi_{2k-1}$ is given by the first row of $S_\sigma^{k}$. Since $\pi_{2k}$
	is the result of substituting all tall triangles in $\pi_{2k-1}$
	according to the substitution rule, the number of tall and short triangles
	in $\pi_{2k}$ is given by the first row of 
	\[
	K_{k}=S_\sigma^{k}\left(\begin{array}{cc}
	2 & 1\\
	0 & 1
	\end{array}\right).
	\]
	Combining the above we have 
	\[
	\lim\limits _{k\rightarrow\infty}\frac{\left|\left\{ \text{Short triangles}\in\pi_{2k-1}\right\} \right|}{\left|\left\{ \text{Tiles\ensuremath{\in\pi_{2k-1}}}\right\} \right|}=\lim_{k\rightarrow\infty}\frac{\left(S_\sigma^{k}\cdot e_{2}\right)_{1}}{\left(S_\sigma^{k}\cdot\boldsymbol{1}\right)_{1}}=\frac{1}{\varphi+1}
	\]
	and
	\[
	\lim\limits _{k\rightarrow\infty}\frac{\left|\left\{ \text{Short triangles }\in\pi_{2k}\right\} \right|}{\left|\left\{ \text{Tiles\,\ensuremath{\in\pi_{2k}}}\right\} \right|}=\lim_{k\rightarrow\infty}\frac{\left(K_{k}\cdot e_{2}\right)_{1}}{\left(K_{k}\cdot\boldsymbol{1}\right)_{1}}=\frac{2}{2+\varphi},
	\]
	and so the limit of the ratio between the number of short triangles
	in $\pi_{m}$ and the total number of tiles does not exist. It can
	be shown in a similar way that the sequence of volumes of the regions
	covered by short triangles in $\pi_{m}$ does not converge.
\end{example}

\section{\label{sec:Fixed-Scaled-Substitutions}Kakutani sequences generated by commensurable schemes substitution schemes}

The proof of Theorem \ref{Thm: Main result} will be given in two stages. In this section we prove the theorem for the commensurable case.

Let $\sigma$ be an irreducible commensurable multiscale
substitution scheme and let $G_\sigma$ be the graph associated
with an equivalent normalized scheme. Then $G_\sigma$ is commensurable,
and by Lemma \ref{lem: Incommensurable equivalent definitions} the
set of lengths of paths in $G_\sigma$ is uniformly discrete, that is, there
exists $\rho>0$ so that for any two paths $\gamma_{1},\gamma_{2}$
of different lengths in $G_\sigma$ 
\[
\left|l\left(\gamma_{1}\right)-l\left(\gamma_{2}\right)\right|\geq\rho.
\]
It follows that it is possible to slightly rescale the prototiles
of the equivalent normalized scheme while remaining in the same equivalent class of
schemes. By Lemma \ref{lem: Graphs and equivlence} this amounts to
slightly sliding the vertices in $G_\sigma$ along the edges, without losing
the monotonicity of the map from the lengths of paths originating at vertex $i\in\mathcal{V}_\sigma$
in the associated graph to the volumes of tiles in $\mathscr{O}^\sigma_{i}$, described
in Corollary \ref{cor: monotonicity of map between paths lengths and tile volumes}.
This proves the following Lemma \ref{lem: Commensurable scheme equiv to rational scheme}.
\begin{lem}
	\label{lem: Commensurable scheme equiv to rational scheme}Let $\sigma$ be an irreducible commensurable
	multiscale substitution scheme. Then $\sigma$ is equivalent to a scheme $\sigma'$ for which
	\begin{enumerate}
		\item Any two edges $\varepsilon_{1}$ and $\varepsilon_{2}$ in the associated
		graph $G_{\sigma'}$ satisfy 
		\[
		\frac{l\left(\varepsilon_{1}\right)}{l\left(\varepsilon_{2}\right)}\in\mathbb{Q}.
		\]
		\item The correspondence between tiles in $\mathscr{O}_i^{\sigma'}$ and finite paths in $G_{\sigma'}$ with original vertex $i\in\mathcal{V}_{\sigma'}$, is monotone in the sense of Corollary \ref{cor: monotonicity of map between paths lengths and tile volumes}, that is, if $\gamma_{T_{1}}$ and $\gamma_{T_{2}}$ are the paths in $G_{\sigma'}$ corresponding to the tiles $T_{1},T_{2}\in\mathscr{O}^{\sigma'}_{i}$, then 
		\[
		\vol T_{1}<\vol T_{2}\,\,\,\,\,\text{{\rm \emph{if\,and\,only\,if}}}\,\,\,\,\,l\left(\gamma_{T_{1}}\right)>l\left(\gamma_{T_{2}}\right).
		\]
	\end{enumerate}
\end{lem}

\begin{thm}
	\label{thm: commensurable can be covered by fixed scale}Let $\sigma$ be an irreducible commensurable
	multiscale substitution scheme and let $\left\{ \pi_{m}\right\} $
	be a Kakutani sequence of partitions of $T_i\in\tau_\sigma$. Then there exist a set
	of prototiles $\tau_{\tilde{\sigma}}$ that contains $\tau_\sigma$ as a subset, and a fixed scale substitution scheme $\tilde{\sigma}$
	with prototile set $\tau_{\tilde{\sigma}}$, so that $\left\{ \pi_{m}\right\} $ is a subsequence
	of a generation sequence of partitions $\{\delta_k\}$ of $T_i\in\tau_{\tilde{\sigma}}$,
	generated by the fixed scale scheme $\tilde{\sigma}$.
\end{thm}

\begin{proof}
	The stages of the process described in the proof are illustrated in
	Example \ref{example: rhombus and triangle} below. 
	
	Put $\sigma_1=\sigma$, and let $G_{\sigma_1}$ be the associated graph. By Lemma
	\ref{lem: tiles =00003D paths}, there is a one-to-one correspondence
	between tiles in the partition $\pi_{m}$ and metric paths of length
	$l_{m}$ in $G_{\sigma_2}$, which is the graph associated with the equivalent
	normalized scheme $\sigma_2$, and has the same set of closed path lengths as
	$G_{\sigma_1}$. By Lemma \ref{lem: Commensurable scheme equiv to rational scheme},
	the normalized scheme is equivalent to a scheme $\sigma_3$ with an associated
	graph $G_{\sigma_3}$ with rationally dependent edge lengths and the same
	set of closed path lengths as $G_{\sigma_1}$ and $G_{\sigma_2}$, and for which
	the statement of Lemma \ref{lem: tiles =00003D paths} still holds.
	Let $\tau_{\sigma_3}=\left(T_{1},\ldots, T_{n}\right) $
	be the list of prototiles on which the ``rationalized'' scheme $\sigma_3$ is defined.
	
	Assume there is an edge of length $\log\frac{1}{\alpha}$ in $G_{\sigma_3}$,
	and let $a\in\mathbb{N}$ be the minimal positive integer such that
	for every $\varepsilon\in G_{\sigma_3}$ there exists $b_{\varepsilon}\in\mathbb{N}$
	coprime to $a$ with 
	\[
	l\left(\varepsilon\right)=\frac{b_{\varepsilon}}{a}\log\frac{1}{\alpha}.
	\]
	It follows that there exists an increasing sequence $\left\{ k_{m}\right\} \subset\mathbb{N}$
	such that $l_{m}=\frac{1}{a}\log\frac{1}{\alpha}\cdot k_{m}$. 
	
	By appropriately adding vertices to $G_{\sigma_3}$, define a new graph $G_{\sigma_4}$
	with edges all of length $\frac{1}{a}\log\alpha$. The graph $G_{\sigma_4}$
	is associated with a fixed scale substitution scheme $\sigma_4$ with contracting
	constant $\alpha^{\frac{1}{a}}$ defined on a prototile set $\tau_{\sigma_4}$,
	that contains rescaled copies of elements of $\tau_\sigma$ and new
	prototiles associated with the new vertices. 
	
	Let $\varepsilon$ be an edge in $G_{\sigma_3}$ with initial and terminal
	vertices $i$ and $j$, and let $v_{1},\ldots,v_{t}$ be the new vertices
	in $G_{\sigma_4}$ added on $\varepsilon$, then they each have a single
	outgoing edge in $G_{\sigma_4}$. It follows that the substitution tiles of
	the associated new prototiles in $\tau_{\sigma_4}$ are geometrically
	trivial, that is $\omega_{\sigma_4}(T_{v_{s}})=\alpha^{\frac{1}{a}} T_{v_{s+1}}$
	and $\omega_{\sigma_4}(T_{v_{t}})=\alpha^{\frac{1}{a}} T_{j}$, and
	so they are all rescaled copies of $T_{j}$. The rescaled
	copy of $T_{j}$ in the substitution rule on $T_{i}$
	that corresponds to the edge $\varepsilon$ in the original scheme,
	is replaced in the fixed scale scheme by a copy of $\alpha^{\frac{1}{a}} T_{v_{1}}$. 
	
	The partition $\delta_{k}$ in the generation sequence generated by the
	fixed scale scheme $\sigma_4$ on $\tau_{\sigma_4}$ defined above corresponds to
	paths in $G_{\sigma_4}$ of length $\frac{1}{a}\log\alpha\cdot k$, which
	can be regarded as metric paths of the same length in $G_{\sigma_3}$. Since
	$\pi_{m}$ corresponds to metric paths of length $l_{m}=\frac{1}{a}\log\frac{1}{\alpha}\cdot k_{m}$,
	by the construction above $\pi_{m}$ is geometrically identical to $\delta_{k_{m}}$. Putting $\tilde{\sigma}=\sigma_4$ finishes the proof.
\end{proof}
\begin{proof}[Proof of Theorem \ref{Thm: Main result} for commensurable schemes]
	By Theorem \ref{thm: commensurable can be covered by fixed scale}
	a Kakutani sequence of partitions generated by a commensurable scheme
	is a subsequence of a generation sequence of partitions generated
	by a fixed scale scheme, which by Theorem \ref{Thm: main result for generation sequences}
	is uniformly distributed, finishing the proof.
\end{proof}
\begin{example}
	\label{example: rhombus and triangle}Consider the fixed scale substitution
	scheme on the triangle $\mathcal{T}$ and the rhombus $\mathcal{R}$
	in $\mathbb{R}^{2}$, as defined in Figure \ref{fig:TR fixed scaled}
	where it is shown together with its associated graph $G_{\sigma_1}$.
	\begin{figure}[H]
		\includegraphics[scale=1]{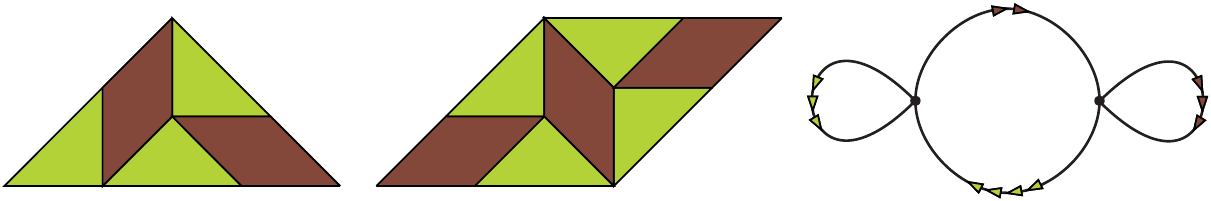}\caption{\label{fig:TR fixed scaled}A fixed scale substitution scheme $\sigma_1$ on $\mathcal{T}$
			and $\mathcal{R}$ and its associated graph $G_{\sigma_1}$. The vertex on
			the left corresponds to $\mathcal{T}$ and the vertex on the right
			corresponds to $\mathcal{R}$. Multiple arrows stand for multiple
			edges of the same length.}
	\end{figure}
	
	This substitution scheme is studied in Chapter 6 of \cite{BaakeGrimm}.
	The substitution matrix of this scheme is given by 
	\[
	S_\sigma=\left(\begin{array}{cc}
	3 & 2\\
	4 & 3
	\end{array}\right).
	\]
	The matrix $S$ has Perron-Frobenius eigenvalue $\mu=\left(1+\sqrt{2}\right)^{2}$
	and $v=\left(1,\sqrt{2}\right)$ is a right Perron-Frobenius eigenvector. The contraction constant of the scheme is thus $\alpha=\frac{1}{1+\sqrt{2}}$,
	the volumes of the prototiles are $\vol \mathcal{T}=1$ and $\vol\mathcal{R}=\sqrt{2}$
	and the edges of the associated graph $G_{\sigma_1}$ are all of length $\log\frac{1}{\alpha}$. 
	
	An equivalent normalized scheme, in which $\vol \mathcal{T}=\vol\mathcal{R}=1$,
	is illustrated next to its associated graph $G_{\sigma_2}$ in Figure \ref{fig: TR normalized}.
	Note that the contraction of the rhombus $\mathcal{R}$ corresponds
	to the sliding of its associated vertex backwards along the edges,
	as described in Lemma \ref{lem: Graphs and equivlence}.
	\begin{figure}[H]
		\includegraphics[scale=1]{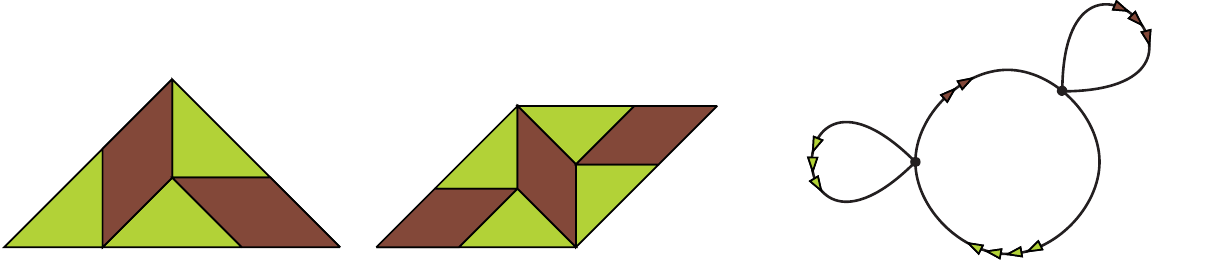}\caption{\label{fig: TR normalized}By contracting the rhombus $\mathcal{R}$
			we obtain an equivalent normalized scheme $\sigma_2$ with an associated graph
			$G_{\sigma_2}$. }
	\end{figure}
	
	By contracting the rhombus $\mathcal{R}$ some more, and continuing
	to slide the associated vertex along the edges accordingly, we obtain
	an equivalent scheme $\sigma_3$ on a set of prototiles $\tau_{\sigma_3}$ with an
	associated graph $G_{\sigma_3}$, which has the properties described in Lemma
	\ref{lem: Commensurable scheme equiv to rational scheme}. This is
	illustrated in Figure \ref{fig:TR rationalized}. 
	\begin{figure}[H]
		\includegraphics[scale=1]{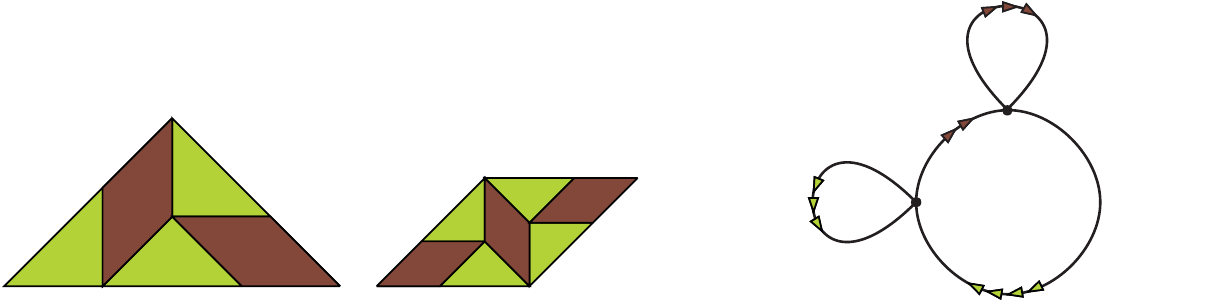}\caption{\label{fig:TR rationalized}An equivalent ``rationalized'' scheme
			$\sigma_3$ on a set of prototiles $\tau_{\sigma_3}$ and the associated graph $G_{\sigma_3}$.}
	\end{figure}
	
	The lengths of the edges in the graph $G_{\sigma_3}$ are all integer multiples
	of $\frac{1}{2}\log\alpha$. By adding vertices to the graph we define
	a new graph $G_{\sigma_4}$ in which all edges are of equal length $\frac{1}{2}\log\alpha$.
	Note that in this example there are $k\geq2$ edges of the same length
	and initial and terminal vertices and so we can choose one of these
	edges, define new vertices as described in the proof of Theorem \ref{thm: commensurable can be covered by fixed scale},
	and define $k$ new edges between the initial vertex and the first
	new one. This is illustrated in Figure \ref{fig: TR rational graph},
	and may be useful for computations.
	
	\begin{figure}[H]
		\includegraphics[scale=1]{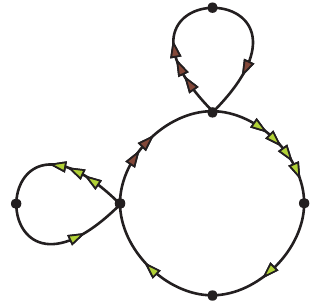}\caption{\label{fig: TR rational graph}The graph $G_{4}$ associated with
			a fixed scale substitution scheme with generation sequences of partitions
			geometrically identical to the Kakutani sequences of partitions of
			the original fixed scale substitution scheme on $\mathcal{T}$ and
			$\mathcal{R}$. }
	\end{figure}
	
	The procedure by which the graph $G_{\sigma_4}$ is derived from the graph
	$G_{\sigma_3}$ corresponds to the addition of four rescaled copies of the
	original prototiles to the set of prototiles $\tau_{\sigma_3}$ illustrated
	in Figure \ref{fig:TR rationalized}. The graph $G_{\sigma_4}$ is the graph
	associated with the fixed scale substitution scheme $\sigma_4$  defined on the extended
	set of prototiles $\tau_{\sigma_4}$, with substitution rule as shown
	in Figure \ref{fig: TR prototiles with additional tiles} and contraction
	constant $\sqrt{\alpha}$. The generation sequence generated by this
	scheme on $\tau_{\sigma_4}$ is geometrically identical to the Kakutani
	sequence generated by the original scheme. 
	\begin{figure}[H]
		\includegraphics[scale=1]{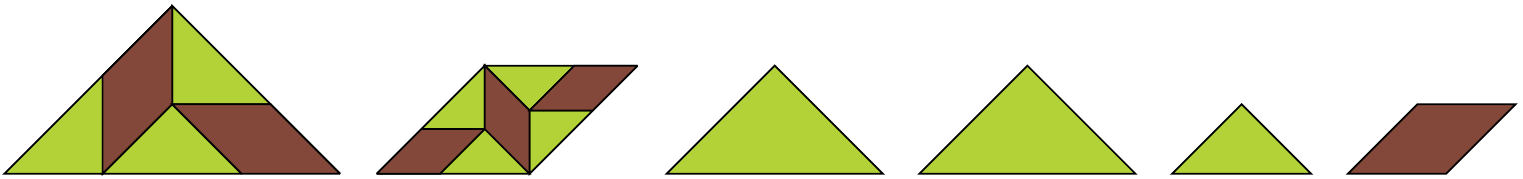}\caption{\label{fig: TR prototiles with additional tiles}The fixed scale substitution
			scheme $\sigma_4$ with contraction constant $\sqrt{\alpha}$, and the set of prototiles $\tau_{\sigma_4}=\left( \mathcal{T},\sqrt{\alpha}\mathcal{R},\sqrt{\alpha} \mathcal{T},\sqrt{\alpha} \mathcal{T},\alpha \mathcal{T},\alpha\mathcal{R}\right)$ it is defined on.}
	\end{figure}
	
\end{example}

\section{\label{sec:Kakutani incommensurable}Kakutani sequences generated	by incommensurable schemes }

In this section we prove Theorem \ref{Thm: Main result} on uniform distribution of Kakutani sequences of partitions generated by incommensurable multiscale
substitution schemes, and present a proof for Theorem \ref{thm:Types, u.d and frequencies}. 

\subsection{Counting metric paths in incommensurable graphs}

In the previous sections, we have seen how the Perron-Frobenius theorem is used to imply uniform distribution of sequences of partitions generated by commensurable schemes. These are no longer sufficient in the incommensurable case, and a new tool for counting tiles is required. As we will see below, the role of the Perron-Frobenius theorem is taken by the path counting results on incommensurable graphs that appear as Theorems $1$ and $2$ in \cite{Graphs}. 

\begin{thm*}[Theorem $1$ in \cite{Graphs}]
	\label{Counting paths main result}Let $G$ be a strongly connected
	incommensurable graph with a set of vertices $\mathcal{V}=\left\{ 1,\ldots,n\right\} $.
	There exist a positive constant $\lambda$ and a matrix $Q^{(M)}\in M_{n}(\mathbb{R})$
	with positive entries, so that if $\varepsilon\in\mathcal{E}$ is
	an edge in $G$ with initial vertex  $h\in\mathcal{V}$, then the
	number of metric paths of length exactly $x$ from vertex $i\in\mathcal{V}$
	to a point on the edge $\varepsilon$ grows as 
	\[
	\frac{1-e^{-l\left(\varepsilon\right)\lambda}}{\lambda}Q_{ih}^{(M)}e^{\lambda x}+o\left(e^{\lambda x}\right),\quad x\rightarrow\infty.
	\]
	The constant $\lambda$ is the maximal real value for which the spectral
	radius of $M(\lambda)$ is equal to $1$, and 
	\[
	Q^{(M)}=\frac{\adj\left(I-M(\lambda)\right)}{-\tr\left(\adj\left(I-M(\lambda)\right)\cdot M^{\prime}(\lambda)\right)},
	\]
	where $M$ is the graph matrix function defined in Section \ref{sec:Graphs},
	$M^{\prime}$ is the entry-wise derivative of $M$, and $\adj A$
	is the adjugate or classical adjoint matrix of $A$, that is, the transpose
	of its cofactor matrix.
\end{thm*}
The proof involves the analysis of the Laplace transform of the corresponding
counting function and the application of the Wiener-Ikehara Tauberian
theorem. For a complete proof see \cite{Graphs}, where further details
regarding walks on directed weighted graphs may be found.

Although for general graphs the value of the constant $\lambda$ 
and therefore also of the entries of the matrix $Q^{(M)}$ appearing in Theorem $1$ in
\cite{Graphs} may be difficult to compute, they can be given explicitly
in the case of graphs associated with multiscale substitution schemes, as follows from Lemma \ref{lem: lambda and Q for graphs associated with schemes}
\begin{lem}\label{lem: lambda and Q for graphs associated with schemes}
	Let $\sigma$ be an irreducible incommensurable substitution scheme in $\mathbb{R}^{d}$, and let $G_\sigma$ be the associated graph. Then $\lambda=d$, and the columns of $Q^{(M)}$ are spanned by $v_{\vol}$, that is, there exist $q_{1},\ldots,q_{n}>0$
	such that $q_{h}\cdot v_{\text{vol}}$ is column $h$ of the matrix
	$Q^{(M)}$.
\end{lem}

\begin{proof}
	By Corollary \ref{cor: The Perron-Frobenius eigenvalue and eigenvector of M(d)},
	the spectral radius of $M(d)$ is exactly $1$. For any
	$x>d$ the sums of all rows of the matrix $M(x)$ are strictly
	smaller than $1$, and so the Perron-Frobenius eigenvalue of $M(x)$
	is strictly smaller than $1$ (see page 63 of \cite{Gantmacher}). It
	follows that the spectral radius of $M(x)$ is strictly
	smaller than $1$ for all $x>d$, and so $\lambda=d$. 
	
	As proved in \cite{Graphs}, the columns of $Q^{(M)}$
	are positive multiples of a right Perron-Frobenius eigenvector of
	$M(d)$. By Corollary \ref{cor: The Perron-Frobenius eigenvalue and eigenvector of M(d)}
	this vector can be chosen to be $v_{\vol}$.
\end{proof}

Given an irreducible incommensurable scheme $\sigma$ in $\R^d$, consider the graph associated with an equivalent normalized scheme, and let $M$ be the corresponding graph matrix function. We denote $M_\sigma=M(d)$, $M'_\sigma=M'(d)$ and $Q_\sigma=Q^{(M)}$.

\begin{cor}
	\label{cor: counting paths in graph associated with scheme}Let $G_\sigma$
	be a graph associated with a normalized irreducible incommensurable
	multiscale substitution scheme in $\mathbb{R}^{d}$. Let $\varepsilon\in\mathcal{E}_\sigma$
	be an edge associated with the tile $\alpha T_{j}\in\omega_\sigma(T_h)$,
	that is $\varepsilon$ has initial vertex $h$, terminal edge $j$
	and is of length $\log\frac{1}{\alpha}$. Then
	the number of metric paths of length exactly $x$ from vertex $i$
	to a point on the edge $\varepsilon$ grows as
	\[
	\frac{1-\beta^{d}}{d}q_{h}e^{dx}+o\left(e^{dx}\right),\quad x\rightarrow\infty,
	\]
	where $\beta$ is the constant of substitution associated with $\alpha$, and $q_{h}\cdot v_{\vol}=q_{h}\cdot\boldsymbol{1}$ is column $h$
	of the matrix $Q_\sigma$.
\end{cor}

\subsection{Tiles in partitions generated by incommensurable multiscale schemes}

In the language of tiles and Kakutani sequences of partitions generated
by multiscale substitution schemes, as a result of Lemma \ref{lem: tiles =00003D paths},
Corollary \ref{cor: counting paths in graph associated with scheme}
amounts to the following result. 
\begin{thm}
	\label{thm: graph path counting in the case of multiscale}Let $\sigma$ be an irreducible incommensurable multiscale substitution scheme in $\mathbb{R}^{d}$ and let $\left\{ \pi_{m}\right\}$
	be a Kakutani sequence of partitions of $T_{i}\in\tau_\sigma$. Denote 
	\[
	b_{hj}:=\sum_{k=1}^{k_{hj}}\frac{1-\left(\beta_{hj}^{(k)}\right)^{d}}{d}.
	\]
	Then the number of tiles of type $j$ appearing in the partition $\pi_{m}$
	of $ T_{i}$ grows as
	\[
	\left|j_{m}\right|=\sum_{h=1}^{n}b_{hj}q_{h}e^{dl_{m}}+o\left(e^{dl_{m}}\right),\quad m\rightarrow\infty,
	\]
	independent of $i$, where $q_{h}\cdot\boldsymbol{1}$ is column $h$ of the matrix $Q_\sigma$.
\end{thm}

\begin{proof}[Proof of Theorem \ref{Thm: Main result} for incommensurable schemes]
	Let $T\in\mathscr{T}^\sigma_{i}$ and assume $T$ is a tile of type $i_{0}$
	that appears in partition $\pi_{m_{0}}$. The path $\gamma_{T}$
	corresponding to $T$ in the graph $G_\sigma$ associated with an equivalent
	normalized scheme terminates at vertex $i_{0}$ and is of length $l\left(\gamma\right)=l_{m_{0}}$,
	and therefore
	\[
	\vol T=e^{-l_{m_{0}}d}\cdot\vol T_{i_{0}}.
	\]
	Let $m>m_{0}$, then by Theorem \ref{thm: graph path counting in the case of multiscale}
	
	\[
	\frac{\left|x_{m}\cap T\right|}{\left|x_{m}\right|}=\frac{\sum\limits _{h,j=1}^{n}b_{hj}q_{h}e^{d\left(l_{m}-l_{m_{0}}\right)}}{\sum\limits _{h,j=1}^{n}b_{hj}q_{h}e^{dl_{m}}}+o\left(1\right),
	\]
	where $q_{h}\cdot\boldsymbol{1}$ is column $h$ of $Q_\sigma$.
	Since the scheme is normalized, $\frac{\vol T_{i_{0}}}{\vol T_{i}}=1$,
	and so 
	\[
	\begin{aligned}\frac{\sum\limits _{h,j=1}^{n}b_{hj}q_{h}e^{d\left(l_{m}-l_{m_{0}}\right)}}{\sum\limits _{h,j=1}^{n}b_{hj}q_{h}e^{dl_{m}}} & =\frac{e^{d\left(l_{m}-l_{m_{0}}\right)}}{e^{dl_{m}}}\cdot\frac{\sum\limits _{h,j=1}^{n}b_{hj}q_{h}}{\sum\limits _{h,j=1}^{n}b_{hj}q_{h}}\\
	& =e^{-l_{m_{0}}d}=e^{-l_{m_{0}}d}\frac{\vol T_{i_{0}}}{\vol T_{i}}\\
	& =\frac{\vol T}{\vol T_{i}}.
	\end{aligned}
	\]
	Therefore
	\[
	\lim_{m\rightarrow\infty}\frac{\left|x_{m}\cap T\right|}{\left|x_{m}\right|}=\frac{\vol T}{\vol T_{i}},
	\]
	and combined with Lemma \ref{Lemma: counting implies uniform distribution}, this completes the proof of Theorem \ref{Thm: Main result}.
\end{proof}

\subsection{Frequencies of types}
\begin{proof}[Proof of parts $\left(1\right)$ and $\left(2\right)$ of Theorem
	\ref{thm:Types, u.d and frequencies}]
	For the first part, simply follow the proof of Theorem \ref{Thm: Main result}
	and replace $\sum\limits _{h,j=1}^{n}b_{hj}$ with $\sum\limits _{h=1}^{n}b_{hr}$.
	
	By Theorem \ref{thm: graph path counting in the case of multiscale},
	the ratio between the number of tiles of type
	$r$ in $\pi_{m}$ and the total number of tiles is given by
	\[
	\frac{|x_{m}^{\left(r\right)}|}{\left|x_{m}\right|}=\frac{\sum\limits _{h=1}^{n}b_{hr}q_{h}e^{dl_{m}}+o\left(e^{dl_{m}}\right)}{\sum\limits _{h,j=1}^{n}b_{hj}q_{h}e^{dl_{m}}+o\left(e^{dl_{m}}\right)}=\frac{\sum\limits _{h=1}^{n}b_{hr}q_{h}}{\sum\limits _{h,j=1}^{n}b_{hj}q_{h}}+o\left(1\right),\quad m\rightarrow\infty.
	\]
	Plugging in the definition of $b_{hj}$ finishes the proof.
\end{proof}
For the proof of the third part of Theorem \ref{thm:Types, u.d and frequencies},
an additional result on random walk on graphs is needed. Let $G$
be a directed weighted graph. For any $i\in\mathcal{V}$ and $\varepsilon\in\mathcal{E}$
with initial vertex $i$, denote by $p_{i\varepsilon}>0$ the probability
that a walker who is passing through vertex $i$ chooses to continue
his walk through edge $\varepsilon$, and assume that for any vertex
in $G$, the sum of the probabilities over all edges originating at
that vertex is equal to $1$. Let $\varepsilon_{1},\ldots,\varepsilon_{k_{ij}}$ be the edges in $G$
with initial vertex $i$ and terminal vertex $j$. The \textit{graph
	probability matrix function} $N:\mathbb{C}\rightarrow M_{n}\left(\mathbb{C}\right)$
is defined by
\[
N_{ij}(s)=p_{i\varepsilon_1}e^{-s\cdot l\left(\varepsilon_1\right)}+\cdots+p_{i\varepsilon_{k_{ij}}}e^{-s\cdot l\left(\varepsilon_{k_{ij}}\right)},
\]
and if $i$ is not connected to $j$ by an edge, put $N_{ij}(s)=0$.

\begin{thm*}[Special case of Theorem $2$ in \cite{Graphs}]
	\label{Main result 2} Let $G$ be a strongly connected incommensurable
	graph with a set of vertices $\mathcal{V}=\left\{ 1,\ldots,n\right\} $,
	and consider a walker on $G$ advancing at constant speed $1$, obeying
	the probabilities $p_{i\varepsilon}$ attached to the edges of $G$.
	There exists a matrix $Q\in M_{n}(\mathbb{R})$ with positive
	entries such that for any $\varepsilon\in\mathcal{E}$ with initial
	vertex $h\in\mathcal{V}$, the probability that a walker who has left
	vertex $i\in\mathcal{V}$ at time $t=0$ is on the edge $\varepsilon\in\mathcal{E}$
	at time $t=T$ is 
	\[
	p_{h\varepsilon}l\left(\varepsilon\right)Q_{ih}+o(1),\quad T\rightarrow\infty.
	\]
	Here 
	\[
	Q^{(N)}=\frac{\adj\left(I-N(0)\right)}{-\tr\left(\adj\left(I-N(0)\right)\cdot N^{\prime}(0)\right)}.
	\]
\end{thm*}
\begin{proof}[Proof of part $\left(3\right)$ of Theorem \ref{thm:Types, u.d and frequencies}]
	The probability that a point in $ T_{i}$ is in a tile
	$\alpha\cdot\varphi\left( T_{j}\right)\in \varrho_\sigma(T_i)$
	after the application of the substitution rule once, is given by
	\[
	\frac{\vol\left(\alpha\cdot\varphi\left( T_{j}\right)\right)}{\vol T_{i}}=\beta^{d}.
	\]
	Therefore, if $\varepsilon$ is the edge associated
	with $\alpha T_{j}\in \omega_\sigma(T_i)$
	in the graph $G_\sigma$ associated with an equivalent normalized scheme,
	then we set
	\[
	p_{i\varepsilon}=\alpha^{d}=\beta^{d}.
	\]
	
	By Lemma \ref{lem: constants and volumes}, the sum of the probabilities
	on all edges with a common initial vertex equals $1$. Observe that
	in the case studied here $N(0)=M(d)=M_\sigma$ and $N^{\prime}(0)=M^{\prime}(d)=M'_\sigma$,
	where $M$ and $N$ are the graph and
	the graph probability matrix functions of $G_\sigma$, and so $Q^{(N)}=Q^{(M)}=Q_\sigma$.
	The result now follows from the special case
	of Theorem $2$ in \cite{Graphs} and from the previous discussion. 
\end{proof}

	Since the formulas for the path counting functions on incommensurable
	graphs that are given in \cite{Graphs} are stated separately for
	every edge in the graph, they imply results which are more refined
	than those stated in Theorem \ref{thm:Types, u.d and frequencies}.
	A nice consequence of this more refined version of Theorem \ref{thm:Types, u.d and frequencies}
	is given in Example \ref{example: red-blue kakutani statistics}.

\begin{example}
	\label{example: red-blue kakutani statistics}Let $\left\{ \pi_{m}\right\} $
	be the $\frac{1}{3}$-Kakutani sequence of partitions of $\mathcal{I}$
	as illustrated in Figure \ref{fig:a few elements in Kakutani third partition}.
	Whenever a partition is made, color shorter of the two new intervals
	red and the longer one black. The first non-trivial elements of this
	colored sequence of partitions is illustrated in Figure \ref{fig: Kakutani red blue}. 
	
	\begin{figure}[H]
		\includegraphics[scale=1.94]{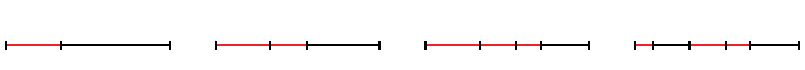}\caption{\label{fig: Kakutani red blue}After each substitution, the shorter
			of the new intervals is colored red.}
	\end{figure}
	
	Then by the arguments used in the proof of parts $\left(2\right)$
	and $\left(3\right)$ of Theorem \ref{thm:Types, u.d and frequencies}, we have
	\[
	\lim_{m\rightarrow\infty}\frac{\left|\left\{ \text{Red intervals \ensuremath{\in}\ensuremath{\ensuremath{\pi_{m}}}}\right\} \right|}{\left|\left\{ {\rm \text{Intervals \ensuremath{\in}\ensuremath{\ensuremath{\pi_{m}}}}}\right\} \right|}=\frac{2}{3}
	\]
	and 
	\[
	\lim_{m\rightarrow\infty}\vol\left(\bigcup\left\{ \text{Red intervals \ensuremath{\in}\ensuremath{\ensuremath{\pi_{m}}}}\right\} \right)=\frac{\frac{1}{3}\log\frac{1}{3}}{\frac{1}{3}\log\frac{1}{3}+\frac{2}{3}\log\frac{2}{3}}.
	\]
\end{example}

\end{document}